\documentclass[11pt,arial]{article}

\usepackage{latexsym}
\usepackage{amsmath}
\usepackage{pstricks}
\usepackage{amssymb}

\usepackage{amsmath}
\usepackage{amsfonts}
\usepackage{amsthm,array} 
\usepackage{array} 
\usepackage{graphicx}
\usepackage{txfonts}
\usepackage{pxfonts}
\usepackage{multirow}  
\usepackage{float}  
\usepackage{longtable}
\usepackage{arydshln}
\usepackage{todonotes}

\usepackage[normalem]{ulem}

\def\br{{\mathbb{R}}}

\def\q{\quad}

\usepackage{color}

\renewcommand{\Re}{\operatorname{Re}}

\newtheorem{thm}{Theorem}[section]
\newtheorem{lem}[thm]{Lemma}

\newtheorem{proposition}[thm]{Proposition}
\newtheorem{Def}[thm]{Definition}
\newtheorem{rmk}[thm]{Remark}

\newtheorem{ex}[thm]{Example}

\newcommand{\END}{\hfill\mbox{\raggedright$\Diamond$}}

\def\lam{\lambda}

\def\a{\alpha}

\def\Del{\Delta}

\def\sign{\text{\,sign\,}}

\usepackage{xcolor}

\usepackage[normalem]{ulem}

\title{ Classification of Codimension-$1$ Singular Bifurcations in Low-dimensional DAEs 
} 
\author{  
I. Ovsyannikov\footnote{MARUM, University of Bremen, Germany. E-mail: 	ivan.i.ovsyannikov@gmail.com} 	\q and\q  H. Ruan\footnote{Technical University of Hamburg, Germany. E-mail: haibo.ruan@tuhh.de}
}

\date{\today}

\begin{document}

\maketitle
\begin{abstract}	
The study of bifurcations of differential-algebraic equations (DAEs) is the topic of interest for many applied sciences, such as electrical engineering, robotics, etc. While some of them were investigated already, the full classification of such bifurcations has not been done yet. 
In this paper, we consider bifurcations of quasilinear DAEs with a singularity and provide a full list of all codimension-one bifurcations in lower-dimensional cases. Among others, it includes
singularity-induced bifurcations (SIBs), which occur when an equilibrium branch intersects a singular manifold causing certain eigenvalues of the linearized problem to diverge to infinity. For these and other bifurcations, we construct the normal forms, establish the non-degeneracy conditions and give a qualitative description of the dynamics. Also, we study singular homoclinic and heteroclinic bifurcations, which were not considered before.
\end{abstract}
\section{Introduction} 

Differential-algebraic equations (DAEs) play an important role in dynamical system modeling such as power systems (cf. \cite{K_1994,A_2003}), nonlinear-circuits (\cite{S_1972,R_2008}), robotics (cf. \cite{BG_2018}), flight control systems (\cite{Venk_1995}), multi-body systems (\cite{EF_1998}), numeric PDEs (\cite{AS_2000} and references in  \cite{KM_2006}). 

We consider {\it quasilinear} DAEs of form
\begin{equation}\label{eq:qdae}
A(x,\a)\dot x=f(x, \a),\q (x,\a)\in \br^n\times \br^m,
\end{equation}
for smooth functions $A:\br^n\times \br^m\to \br^{n\times n}$, $f:\br^n\times \br^m\to \br^n$ of the phase variable $x\in \br^n$ and the parameter $\a\in\br^m$. 

In the presence of {\it singularities}, that is, points $(x,\a)$ such that $\det A(x,\a)$ in (\ref{eq:qdae}) vanishes, it is not possible to describe the local behavior of a DAE in terms of an explicit ODE. Regularization of such a singular DAE often leads to an ODE with higher dimensional manifolds of equilibria in phase space, which can manifest {\it bifurcations without parameters} (cf. \cite{Riaza_2018}). 

In parametrized problems, a stability change  due to the divergence of an eigenvalue was first analyzed by Venkatasubramanian \cite{Venk_1992} and later addressed by many others (\cite{Venk_1995, B_1998, B_2000, B_2001, Riaza_2002}). The change of stability, termed {\it singularity-induced bifurcations} (SIB), occurs when an equilibrium branch intersects a singular manifold, which results in the divergence of at least one eigenvalue through infinity. 

Main efforts in studying such singularity-crossing equilibria have been given by trying to characterize the SIBs in terms of the {\it linearized} problems, such as using the matrix pencils $\{A(x^*,\a^*),-D_xf(x^*,\a^*)\}$ associated to   (\ref{eq:qdae}) at the point of singularity $(x^*,\a^*)$ (cf. \cite{B_2001,Riaza_2010}). Different sufficient conditions have been given in the framework of the tractability index (cf. \cite{M_2002, T_2003}) and the geometric index (\cite{RR_2002,R_1995,Riaza_2010}) among others. However, they may not provide a {\it necessary and sufficient characterization} of the local flow around such singularity-crossing equilibria.

Besides SIBs, there can be other singular behavior induced by the presence of singularities such as the change of the singularity surface itself ({\it fold}) or bifurcations of singular equilibria that change significantly the dynamics near the singularity. 

Quasilinear DAEs (\ref{eq:qdae}) have a strong connection to another important class of dynamical systems -- fast-slow systems. Indeed, a system 
$$
\varepsilon \dot x = f(x,y), \q \dot y = g(x,y)
$$
is a fast-slow system for small $\varepsilon$.
Setting $\varepsilon$ to zero we obtain the so-called \textit{slow system}:
\begin{equation}\label{slowfast}
0 = f(x, y), \q \dot y = g(x, y),
\end{equation}
in which the first equation defines the slow manifold and the second one determines the dynamics, restricted onto it. System (\ref{slowfast}) is a DAE that can be brought to the form of an ODE system via time-differentiation of the algebraic equation and the substitution of $\dot y$ from the second one:
\begin{equation}\label{qldae}
f_x(x,y) \dot x = - f_y(x,y) g(x,y),  \q \dot y = g(x, y).
\end{equation}
That is, the slow system can be reduced to the form
(\ref{eq:qdae}), so our research here contributes also to the studies of bifurcations in slow-fast systems (see \cite{Ashwin_2020}). The correspondence between the terms and notions can be viewed in the following way: 
\begin{itemize}
    \item  A singularity set in DAEs corresponds to a fold set of a slow manifold;
    \item a fold of the singularity surface corresponds to a cusp of a slow manifold.
\end{itemize}

In this paper, we will call bifurcations that are caused by the presence of singularities {\it singular bifurcations}, which is a more general consideration of possible scenarios, which may or may not involve equilibria directly. That is, we consider bifurcations caused by singularities including SIBs but {\it not exclusively} so. 

To make a clear impression of the realm of all possible bifurcations, we focus on low-dimensional quasilinear DAEs of form (\ref{eq:qdae}) for which $x\in\br$ or $x\in \br^2$. Our goal is to provide a list of all possible singular bifurcations of codimension $1$ in such systems.

The paper is organised as follows. In Section~\ref{sec:1D} the basic notions for the study of one-dimensional quasilinear DAEs are given, all possible codimension-one   bifurcations are studied and the behaviour in higher-codimension bifurcations is described. In Section~\ref{sec:2D} the main notions are given for a two-dimensional case, for which the full list of possible codimension-one bifurcations is provided. Section~\ref{sec:loc_norm} contains the rigorous derivation of the dynamical behaviour near some local singular bifurcations from Section~\ref{sec:2D}.

\section{Quasilinear DAEs: One-dimensional case}\label{sec:1D} 

%




%
Consider a quasilinear DAE (\ref{eq:qdae}) for $n = 1$, it is given by 
\begin{equation}\label{eq:dds_1D}
g(x,\a) \dot x =f(x,\a),\q (x,\a)\in \br\times \br^m,
\end{equation}
for smooth functions  $f:\br\times \br^m\to \br$ and $g:\br\times \br^m\to \br$. Note that in this case, the singular set is precisely the set of zeros of $g$ given by
\begin{equation}\label{eq:ss}
\Sigma_\a=\{x\in \br\, :\, g(x,\a)=0\},
\end{equation}
which under a regularity assumption on $g$, is composed of isolated points. We will call every such point 
a {\em singularity}. The set of zeros of $f$ that are not zeros of $g$
\begin{equation}\label{eq:zs}
E_\a=\{x\in \br\, :\, f(x,\a)=0, \, g(x, \a) \neq 0\},
\end{equation}
will be called the {\em equilibrium set}. 
Under a similar regularity assumption, this set is also composed of isolated points.
Each such point is an equilibrium of system (\ref{eq:dds_1D}). 

\begin{Def}\rm  A point $x\in\br$ is called a {\it singular equilibrium} if it lies in the intersection of the singular set $\Sigma_\a$ given by (\ref{eq:ss}) with the zeros set of $f(x, \a)$ for some $\a\in \br^m$.
\end{Def}

\begin{Def}\label{def_simple}\rm
For a fixed parameter value $\a = \a_0$, we will call the point $x_0$ {\em a simple equilibrium} if it is a simple zero of $f$ and not a zero of $g$:
$$
f(x_0, \a_0) = 0, \, f'_x(x_0, \a_0) \neq 0, \, g(x_0, \a_0) \neq 0. 
$$
Analogously, we will call $x_0$ {\em a simple singularity} if it is a simple zero of $g$ and not a zero of $f$:
$$
f(x_0, \a_0) \neq 0, \, g(x_0, \a_0) = 0, \, g'_x(x_0, \a_0) \neq 0. 
$$ 
\end{Def}

\begin{Def}\rm
A simple singularity point $x_0$ is called {\em incoming}  ({\it outgoing}), if there exists a  small neighbourhood $U$ of $x_0$ such that for any initial condition $x \in U$ the solution $x(t)$ reaches
$x_0$ in finite forward (backward) time. 
\end{Def}

\begin{rmk}\rm  The incoming or outgoing simple singularities are known in the DAE literature as the {\it standard singular points} which was introduced in  \cite{RR_1994}. They behave like an {\it impasse point}, where solutions are no longer defined being either attractive  or repelling (cf. \cite{RR_1994,RZ_2001}).
\end{rmk}

For a simple equilibrium point $x_0$ there exists a small neighbourhood $|x - x_0| < \varepsilon$ such that $g(x, \a_0) \neq 0$,
and the system (\ref{eq:dds_1D}) can be rewritten as 
\begin{equation}\label{eq:ftilde}
\dot x=\frac{f(x,\a_0)}{g(x,\a_0)}:=\tilde f(x,\a_0), \,\text{where}\, \tilde f(x_0,\a_0) = 0.
\end{equation}
Thus, the stability type of the equilibrium $(x_0,\a_0)$ is completely determined by the sign of the derivative
\begin{equation}\label{eq:lam_a}
\lambda := \partial_x \tilde f \big|_{(x_0,\a_0)}=\frac{g \partial_x f - f \partial_x g}{g^2}\big|_{(x_0,\a_0)}=\frac{\partial_x f
}{g}\big|_{(x_0,\a_0)}\ne 0,
\end{equation}
which is non-zero by the assumption that $x_0$
is a simple equilibrium for $\a=\a_0$. If $\lambda < 0$, then  the equilibrium $x_0$
is stable; if $\lambda > 0$, then it is unstable.  See Figure~\ref{F:1d_a}.
\begin{figure}[!htb]
	\centerline{
		\includegraphics[width=1\textwidth]{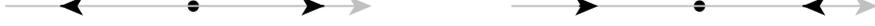}
	} 
	\caption{The local flow around a simple equilibrium 
	for 
	$\lambda > 0$ (unstable, left)
	and for  
	$\lambda < 0$ (stable, right), where $\lam$ is given by (\ref{eq:lam_a}).}
	\label{F:1d_a}
\end{figure}

For a simple singularity $x_0$, there exists a small neighborhood in which $f(x, \a_0) \neq 0$ and the system (\ref{eq:dds_1D}) can be rewritten as
\begin{equation}\label{eq:gtilde}
\tilde g \dot x=1,\,\text{for}\, \tilde g = \frac{g(x,\a)}{f(x,\a)}\, \text{with}\, \tilde g(x_0, \alpha_0) = 0.
\end{equation}
Thus, the following derivative 
\begin{equation}\label{eq:lam_b}
\lambda = \tilde g_x\big|_{(x_0,\a_0)}=\frac{g_xf-gf_x}{f^2}\big|_{(x_0,\a_0)}=\frac{g_x}{f}\big|_{(x_0,\a_0)}\ne 0,
\end{equation}
determines the type of the singularity at $(x_0,\a_0)$. More precisely, if $\lambda > 0$, 
then the simple singularity point $x_0$ is outgoing; if 
$\lambda < 0$, then it is incoming. See Figure~\ref{F:1d_pm}, we put here and further below the double arrow to reflect the fact that the trajectory reaches the singularity in finite time, and the velocity grows to infinity.
 \begin{figure}[!htb]
 	\centerline{
 		\includegraphics[width=1\textwidth]{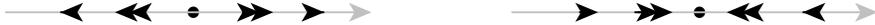}
 	} 
 	\caption{The local flow around a simple singularity point for 	
 	$\lambda > 0$ (left) and for  
 	$\lambda < 0$ (right), where $\lam$ is given by (\ref{eq:lam_b}).}
 	\label{F:1d_pm}
 \end{figure}

It is clear from Definition~\ref{def_simple}, that simple equilibria and simple singularities persist
under generic parametric perturbations. Indeed, the condition $f'_x(x_0, \alpha_0) \neq 0$ implies that by the Implicit Function Theorem, there exists locally a unique function $x^*(\alpha)$ with $x^*(\alpha_0) = x_0$, which fulfills the equation $f(x, \alpha) = 0$. Moreover, this equilibrium maintains
the same stability type, as the exponent $\lambda$ in (\ref{eq:lam_a}) preserves its sign. In a similar way, one can deduct the corresponding property for a simple singularity. 

\begin{thm}\label{thm_1D_struct}
If in an open set $U \subset \mathbb{R}$ system (\ref{eq:dds_1D}) possesses a finite set of  equilibria and a finite set of singularities and all of them are simple, then the system is structurally stable in $U$.
\end{thm}
\begin{proof} Under sufficiently small perturbations, every simple equilibrium and every simple singularity stays in a small neighbourhood of its initial position, remain simple, are distributed in the same order on the line and keep their stability types. Also, neither of these points reaches the boundary of $U$. The intervals bounded by these equilibria and singularities can be homeomorphically conjugated, with the direction of motion preserved.
\end{proof}

However, when the conditions of Theorem~\ref{thm_1D_struct} are violated, one  may encounter {\em bifurcations}. There are three such possibilities {of singular bifurcations}\footnote{We do not consider cases of infinite sets of equilibria or singularities, as it violates the assumption that functions $f$ and $g$ are  smooth.}:
\begin{itemize}
    \item[A1.] A non-simple equilibrium: $f(x,\alpha) = 0$, $f'_x(x,\alpha) = 0$;
    \item[A2.] A non-simple singularity: $g(x,\alpha) = 0$, $g'_x(x,\alpha) = 0$;
    \item[A3.] A singular equilibrium: $f(x,\alpha) = 0$, $g(x,\alpha) = 0$.
\end{itemize}

These cases are not exclusive to each other. They can happen simultaneously, either at the same or different points, which increases the codimension of the problem. In the following, we  assume that the bifurcation conditions occur at $a_0 = 0$ and $x_0 = 0$, which can be achieved by appropriate translation of coordinates and parameters.

We start with formulating the simplest possible cases, i.e. the cases of {\it codimension-$1$}. 
\begin{itemize}
    \item[A1.1.] A codimension one non-simple equilibrium: $f(0,0) = 0$, $f'_x(0,0) = 0$, $f''_x(0, 0) \neq 0$, $g(0,0) \neq 0$;
    \item[A2.1.] A codimension one non-simple singularity: $g(0,0) = 0$, $g'_x(0,0) = 0$, $g''_x(0, 0) \neq 0$, $f(0,0) \neq 0$; 
    \item[A3.0,0.] A transcritical singularity (codimension-$1$ singular equilibrium\footnote{Two zeros represent that the left- and right-hand sides of the equation are not degenerate. One codimension is added, because both functions vanish at the same point. The generalization of this case, case A3.$m,n$ is given at Page~\pageref{HO_A3}}): $f(0,0) = 0$, $g(0,0) = 0$, $f'_x(0,0) \neq 0$, $g'_x(0,0) \neq 0$. 
\end{itemize}


The case A1.1 is completely analogous to the usual fold bifurcation in dynamical systems without {singularities}. Indeed,
 since $g(x,\alpha) \neq 0$ in a small neighborhood of $(x_0,\a_0)=(0,0)$, we can rewrite the system  (\ref{eq:dds_1D})  using the function $\tilde f$ as defined in (\ref{eq:ftilde}), where 
$$
\tilde f(0, 0) = \tilde f'_x(0, 0) = 0, \ \tilde f''_{xx}(0, 0) = \frac{f''_{xx}}{g}\big|_{(0,0)} \neq 0.
$$

\begin{proposition}\label{prop:nf_equi_fold}
Assume that for system (\ref{eq:dds_1D}) the 
conditions of case A1.1 are fulfilled for $\alpha= 0$ at $x = 0$, and $f'_\alpha(0, 0) \neq 0$. Then for all small $\alpha$ by an invertible change of coordinate and parameter, the system can be brought near the origin to the following normal form:
\begin{equation}\label{nf_equi_fold}
  \dot \eta = \beta +{s}\eta^2 + O(\eta^3),
  \end{equation}
  {where $s=\sign\frac{f_{xx}''}{g}\big|_{(0,0)}$}
  (cf. Figure ~\ref{F:simp_deg}(left)).
\end{proposition}


This result immediately follows from \cite{Kuz}, Theorem~3.2.
Similarly, one can derive the normal form for the non-simple singularity in case~A2.1. 

\begin{proposition}\label{prop:nf_deg_fold}
Assume that for the system (\ref{eq:dds_1D}) the 
conditions of case A2.1 are fulfilled for $\alpha = 0$ at $x = 0$, and $g'_\alpha(0, 0) \neq 0$. Then, for all small $\alpha$ by an invertible change of coordinate and parameter, the system  can be brought near the origin to the following normal form:
\begin{equation}\label{nf_deg_fold}
    (\beta+{s} \eta^2 + O(\eta^3)) \dot \eta = 1,
\end{equation}
  {where $s=\sign\frac{g_{xx}''}{f}\big|_{(0,0)}$} (cf. Figure ~\ref{F:simp_deg}(right)).
\end{proposition}

\begin{proof}
In some small neighbourhood of the origin, we have $f(x, \alpha) \neq 0$ for all small $\alpha$. Then, the system can be rewritten in the form (\ref{eq:gtilde}) with
$$
\tilde g(0, 0) = \tilde g'_x(0, 0) = 0, \ \tilde g''_{xx}(0, 0) = \frac{g''_{xx}}{f}\big|_{(0,0)} \neq 0.
$$
We expand $\tilde g$ in Taylor series in $x$ as
$\tilde g = g_0(\alpha) + g_1(\alpha) x + g_2(\alpha) x^2 + O(x^3)$ with $g_0(0) = g_1(0) = 0$ and $g_2(0) = a \neq 0$. Using  a parameter-dependent coordinate shift of the form $x = y + \delta(\alpha)$ with $\delta(0) = 0$, we can rewrite $\tilde g$ as:
\begin{equation*}
\begin{split}
\tilde g(y + \delta(\alpha), \alpha) = & (g_0(\alpha) + g_1(\alpha) \delta(\alpha) + O(\alpha^2)) + (g_1(\alpha) + 2 g_2(\alpha)\delta(\alpha) + O(\alpha^2)) y \\
& + (g_2(\alpha) + O(\alpha))y^2 + O(y^3),
\end{split}
\end{equation*}
where the linear term can be neglected by an appropriate choice of $\delta$. Indeed, the coefficient of the linear term vanishes at $\alpha = \delta = 0$, and its derivative with respect to $\delta$ at zero
is given by $2a \neq 0$. By the Implicit Function Theorem, there exists a function $\delta(\alpha) = \delta_1 \alpha + O(\alpha^2)$ with $\delta_1 = -\frac{g_{1,\a}'(0)}{2 a}$. Thus, we have
$$
\left[g_{0,\a}'(0) \alpha + O(\alpha^2) + a(\alpha)y^2 +
O(y^3)\right] \dot y = 1
$$
with $a(0) = a$, which becomes (\ref{nf_deg_fold}) using the scaling 
$$
y = \frac{1}{|a(\alpha)|^{1/3}} \eta, \ 
\beta  = \frac{1}{|a(\alpha)|^{1/3}}(g_{0,\a}'(0) \alpha + O(\alpha^2)).
$$
Moreover,
as $g_{0,\a}'(0) = \frac{g'_\alpha}{f}(0, 0) \neq 0$, all the
transformations are invertible.
\end{proof}

The  bifurcation occurs
in the following way, if $s = -1$: for  $\beta = 0$ there exists
a locally unique non-simple singularity point, which
disappears when $\beta < 0$ and is replaced by a pair
of incoming and outgoing simple singularity points when $\beta > 0$. See
Figure~\ref{F:simp_deg}~(right).

 \begin{figure}[!htb]
 	\centerline{
 	\includegraphics[width=1\textwidth]{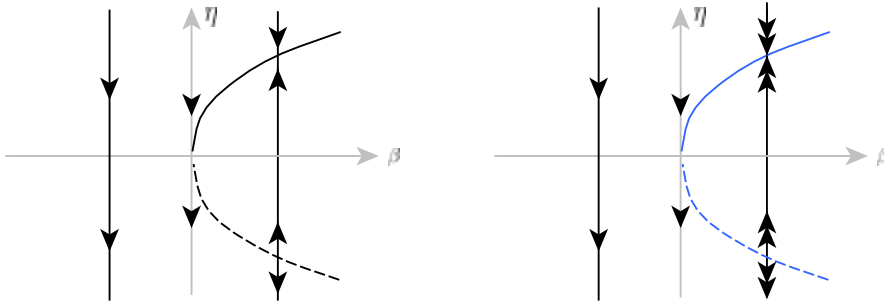}
 	} 
 	\caption{The equilibrium fold bifurcation (left) and the singularity fold bifurcation (right) in normal form (\ref{nf_equi_fold}) and (\ref{nf_deg_fold}), respectively, where we have taken $s = -1$ in both cases.}
 	\label{F:simp_deg}
 \end{figure}

The following Proposition states the normal form of the transcritical singularity bifurcation in case A3.0,0.

\begin{proposition}\label{prop:1d_deg_bif_simple}
If the system (\ref{eq:dds_1D})  satisfies the conditions
of transcritical singularity, case A3.0.0, at $(x, \alpha) = (0, 0)$, and
$$A:=(g_x' f_\a' - f_x' g_\a')\big|_{(0,0)} \neq 0,$$
then there exists an
invertible change of coordinate and parameter,
which brings system (\ref{eq:dds_1D}) near the origin to the normal form
\begin{equation}\label{eq:1d_nf_a}
\eta\dot \eta = \beta + s \eta + O(\eta^2),
\end{equation}
where
$$
s = \sign \frac{f_x'}{g_x'}\big|_{(0,0)}.
$$
(cf. Figure \ref{F:1d_nf_a}).
\end{proposition}

\begin{proof}
As $g_x'(0, 0) \neq 0$, the implicit equation $g(x,\a)=0$ can be locally uniquely resolved with respect to $x$ for small $x$ and $\alpha$. That is, there exists a smooth function $x^*(\alpha)$ such that 
$g(x^*(\alpha), \alpha) \equiv 0$ with $x^*(0) = 0$ 
and $x^*(\alpha) = -\frac{g_\a'}{g_x'}\big|_{(0, 0)} \alpha + O(\alpha^2)$.
Consider a parameter-dependent shift of coordinate 
$x = x^*(\alpha) + y$. Then, the left-hand side of (\ref{eq:dds_1D}) is transformed as
\begin{align}
g(x^*(\alpha) + y, \alpha) &= g_x'(x^*(\alpha),\alpha)y + O(y^2) = 
(g_1 + O(\alpha))y + O(y^2) \notag\\
&=(g_1 + O(\alpha))y(1 + O(y)),\label{eq:transcrit_nf_pf}
\end{align} 
where $g_1 = g_x'(0, 0)\ne 0$. The right-hand side function $f$  becomes
\begin{align}
f(x^*(\alpha) + y, \alpha) &= f(x^*(\alpha), \alpha) +
f_x'(x^*(\alpha), \alpha)y + O(y^2) \notag \\
&=\frac{A} {g_1}\alpha + O(\alpha^2) + (f_1 + O(\alpha)) y + O(y^2),\label{eq:transcrit_f} 
\end{align}
where $f_1 = f_x'(0, 0)$.

We choose the neighbourhood small enough for term 
$(1 + O(y))$ in formula (\ref{eq:transcrit_nf_pf}) to stay always positive. Then, we reparametrize time by formula $dt/(1 + O(y)) = d \tau$, and also 
divide by a non-zero coefficient $(g_1+O(\a))$, system (\ref{eq:dds_1D}) is transformed as:
%
\begin{equation}\label{eq:trans_1}
    y \dot y = \frac{A} {g_1^2}\alpha + O(\alpha^2) + \left(\frac{f_1}{g_1} + O(\alpha)\right) y + O(y^2).
\end{equation}
It leads to (\ref{eq:1d_nf_a}) by scaling $y \to \left|\frac{f_1}{g_1} + O(\alpha)\right| \eta$ and setting 
\begin{equation}\label{1d_A3_1_beta}
\beta = \frac{A} {f_1^2}\alpha + O(\alpha^2).
\end{equation}
Notice that  as $A \neq 0$, the new small parameter $\beta$ diffeomorphically depends on $\alpha$.
\end{proof}

 \begin{figure}[!htb]
 	\centerline{
 		\includegraphics[width=1\textwidth]{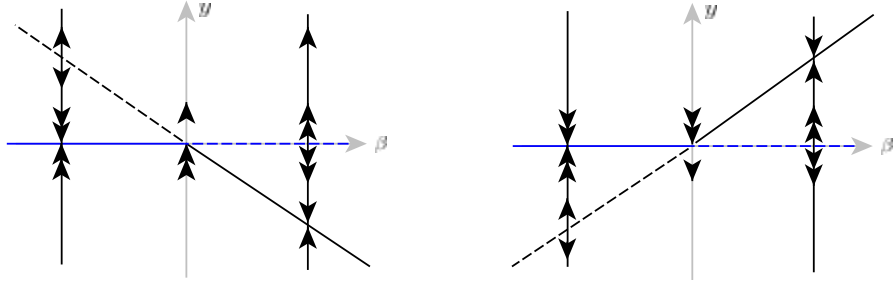}
 	} 
 	\caption{Transcritical singularity bifurcation of (\ref{eq:1d_nf_a}): $s > 0$ (left); $s < 0$ (right). Both feature a transition from an incoming to an outgoing singularity as $\beta$ changes from the negative to the positive, where $\beta$ changes according to (\ref{1d_A3_1_beta}). The dashed (solid) lines indicate unstable (stable) equilibrium and singularity points.}
 	\label{F:1d_nf_a}
 \end{figure}

\begin{ex}\label{ex:example_A1_2_3}\rm The following systems demonstrate examples of cases A1.1, A2.1 and A3.0,0, respectively.
\begin{align}
   & (x+1)\dot x=x^2+\a,\q x,\a\in \br\label{eq:ex_equi_fold}\\
&    (x^2+\a)\dot x=x+1,\q x,\a\in \br\label{eq:ex_dege_fold}\\
 &   (x+x^2+\a)\dot x=x-x^2+2\a,\q x,\a\in \br. \label{eq:ex_transcrit}
\end{align}
\end{ex}

More precisely, consider the system (\ref{eq:ex_equi_fold}) with $f(x, \a) = x^2+\a$ and $g(x, \a) = x + 1$ at $(x,\a)=(0,0)$. It satisfies the non-simple equilibrium conditions, case  A1.1:
\[f(0,0)=f_x'(0,0) =0, \q f''_{xx}(0, 0) = 2 \neq 0,\q g(0,0) = 1 \neq 0, \q f_\a'(0,0) = 1 \neq 0.\]
By Proposition \ref{prop:nf_equi_fold}, the normal form of this bifurcation is given by~(\ref{nf_equi_fold}) with $s=\sign\frac{f_{xx}''}{g}\big|_{(0,0)} = 1$.


 For the system (\ref{eq:ex_dege_fold}) at the non-simple singularity point $(x,\a)=(0,0)$, one has 
\[g(0,0)=g_x'(0,0) = 0, \q g''_{xx}(0, 0) = 2 \neq 0, \q f(0,0) = 1 \neq 0,\q g_\a'(0,0)=1 \neq 0.\]
By Proposition~\ref{prop:nf_deg_fold}, the normal form here is~(\ref{nf_deg_fold}) with $s=\sign\frac{g_{xx}''}{f}\big|_{(0,0)} = 1$.

The system (\ref{eq:ex_transcrit}) at point $(x,\a)=(0,0)$ satisfies the conditions of the transcritical singularity:
\[f(0,0)=g(0,0)=0,\q f_x'(0,0) = 1\neq 0,\q g_x'(0,0)=1 \neq 0, \q \left|\begin{array}{cc}g_x'& g_\a'\\f_x'& f_a'\end{array}\right|_{(0,0)} = 1 \neq 0.\]
By Proposition \ref{prop:1d_deg_bif_simple}, the normal form of this bifurcation is (\ref{eq:1d_nf_a}) with
$s = \sign \frac{f_x'}{g_x'}\big|_{(0,0)} = 1$.

\END

Besides the codimension-$1$  bifurcations listed by A1.1--A3.0,0, one can describe generic unfoldings of bifurcations of higher codimension  in a similar way.
These bifurcations admit invertible changes of coordinates (shifts), that bring them to the corresponding normal forms. We do not give a proof here, because it is straightforward: similar to cases A1.1 and A2.1 there always exists the parameter-dependent shift of coordinates $x \to y + \delta(\a)$ such that the term that vanishes at $\a=0$ and has the highest power in the Taylor expansion ($y^m$ or $y^n$ below), is eliminated for all small $\a$.
We distinguish the following bifurcations:
\begin{itemize}
    \item[A1.$m$.] A codimension-$m$ equilibrium: $f(0,0) = 0$, $f^{(i)}_x(0,0) = 0$ for $1 \le i \le m$ and $f^{(m + 1)}_x(0, 0) \neq 0$, $g(0,0) \neq 0$.
    The normal form is:
    \begin{equation}\label{HO_A1}
        \dot y = \beta_0 + \beta_1 y + \beta_2 y^2 + \ldots + \beta_{m - 1} y^{m - 1} + s y^{m + 1} + O(y^{m + 2})
    \end{equation}

    \item[A2.$n$.] A codimension-$n$ singularity: $g(0,0) = 0$, $g^{(i)}_x(0,0) = 0$ for $1 \le i \le n$ and $g^{(n + 1)}_x(0, 0) \neq 0$, $f(0,0) \neq 0$. The normal form of this bifurcation is:
     \begin{equation}\label{HO_A2}
        ( \alpha_0 + \alpha_1 y + \alpha_2 y^2 + \ldots + \alpha_{n - 1} y^{n - 1} + s y^{n + 1} + O(y^{n + 2}))\dot y = 1
    \end{equation}
    \item[A3.$m,n$.] A codimension-($1 + m + n$) singular equilibrium: $f(0,0) = 0$, $f^{(i)}_x(0,0) = 0$ for $1 \le i \le m$, $g(0,0) = 0$, $g^{(j)}_x(0,0) = 0$ for $1 \le j \le n$
    and $f^{(m + 1)}_x(0, 0) \neq 0$, $g^{(n + 1)}_x(0, 0) \neq 0$. Its normal form is:
    \begin{equation}\label{HO_A3}
    (\a_0+\a_1y+\cdots + \a_{n-1} y^{n-1} + s y^{n+1})\dot y= \beta_0+\beta_1y+\cdots+ \beta_{m}y^{m}+y^{m+1}.
    \end{equation}
\end{itemize}
In the formulas above coefficient $s$ is equal to either $+1$ or $-1$, and all $\alpha_i$ and $\beta_i$ are small unfolding parameters.

\begin{lem}
In one-dimensional system (\ref{eq:dds_1D}) the higher-order bifurcations occur in the  way that under small perturbations the following dynamics is observed, depending on the case A1--A3.

\begin{itemize}
    \item[Case A1.$m$:] 
         For any combination of integers $\{ a_1, a_2, \ldots, a_k\}$, such that $A = \sum\limits_{i = 1}^k a_i \le m + 1$ and $A$ has the same parity with $m + 1$, there exists a small perturbation of normal form (\ref{HO_A1}), such that it has locally $k$ equilibria with coordinates $x_1 < x_2 < \ldots < x_k$, and every $x_i$ is simple if $a_i = 1$ or non-simple of codimension $a_i - 1$, if $a_i > 1$.
    
    \item[Case A2.$n$:] 
    For any combination of integers $\{ b_1, b_2, \ldots, b_l\}$, such that $B = \sum\limits_{j = 1}^l b_i \le  n + 1$ and $A$ has the same parity with $n$, there exists a small perturbation of normal form (\ref{HO_A2}), such that it has locally $l$ singularities with coordinates $y_1 < y_2 < \ldots < y_l$, and every $y_j$ is simple if $b_j = 1$ or non-simple of codimension $b_j - 1$, if $b_j > 1$.
    
    \item[Case A3.$m,n$:] 
    For any two combinations of integers $\{ a_1, a_2, \ldots, a_k\}$ and $\{ b_1, b_2, \ldots, b_l\}$ described above,
    and any two sets of local coordinates: $x_1 < x_2 < \ldots < x_k$ and $y_1 < y_2 < \ldots < y_l$,
    there exists a small perturbation of normal form (\ref{HO_A3}), such that it has:
      \begin{itemize}
       \item[1)] an equilibrium at the point $x_i$, if $x_i$ does not coincide with any of $y_j$; the equilibrium is simple if $a_i = 1$ or non-simple of codimension $a_i - 1$, if $a_i > 1$.
       \item[2)] a singularity at the point $y_j$, if $y_j$ does not coincide with any of $x_i$; the singularity is simple if $b_j = 1$ or non-simple of codimension $b_j - 1$, if $b_j > 1$.
       \item[3)] a singular equilibrium at point $x_i$ if $x_i = y_j$ for some $j$; this point is degenerate of codimension $a_i + b_j + 1$.
      \end{itemize}
\end{itemize}
\end{lem}

\begin{proof}
First, we consider cases A1.$m$ and A2.$n$. Take a set of integers $\{ a_1, a_2, \ldots, a_k\}$ as described above in the respective case, and select small $x_1 < x_2 < \ldots < x_k$ (this means, that $|x_i| < \varepsilon$ for some $\varepsilon$). Then, construct the following  polynomial:
\begin{equation}\label{proof_high}
    P(y) = (y - x_1)^{a_1}(y - x_2)^{a_2} \ldots (y - x_k)^{a_k}.
\end{equation}
It has $k$ roots at coordinates $x_i$ with multiplicity $a_i$ each. If we open the parentheses in formula (\ref{proof_high}), the resulting polynomial will be a small perturbation of the order-$(m + 1)$ polynomial standing in the right-hand side of normal form (\ref{HO_A1}) or the order-$(n + 1)$ polynomial in the left-hand side of normal form (\ref{HO_A2}). The statement of Lemma on equilibria or singularities respectively, directly follows.

Now take case A3.$m,n$ and the respective sets of integers $\{ a_1, a_2, \ldots, a_k\}$ and $\{ b_1, b_2, \ldots, b_l\}$ and coordinates: $x_1 < x_2 < \ldots < x_k$ and $y_1 < y_2 < \ldots < y_l$. We construct two polynomials:
\begin{equation}\label{proof_high_1}
   \begin{array}{l}
    P(y) = (y - x_1)^{a_1}(y - x_2)^{a_2} \ldots (y - x_k)^{a_k}, \\ 
    Q(y) = (y - y_1)^{b_1}(y - y_2)^{b_2} \ldots (y - y_l)^{b_l}.
    \end{array}
\end{equation}
Polynomial $P(y)$ is the small perturbation of the right-hand side and polynomial $Q(y)$ is the small perturbation of the left-hand side of normal form (\ref{HO_A3}). At the same time, this system possesses equilibria, singularities and singular equilibria exactly as described in the Lemma.
\end{proof}

\section{Quasilinear DAEs: Two-dimensional case}\label{sec:2D}

Consider the two-dimensional quasilinear DAEs of form (\ref{eq:qdae}), where $A$ is  everywhere nonsingular except on the singular set  $\Sigma_{\a}$. The simplest possible form of such DAEs is given by (cf. \cite{B_2002}) 
\begin{equation}\label{eq:dds_2D_g}
\begin{cases}
g(x,y,\a)\dot x=f_1(x,y,\a)\\
\dot y=f_2(x,y,\a), 
\end{cases}
\end{equation}
for $(x,y)\in \br^2$, $\a\in \br^m$ and $g:\br^2\times \br^m \to \br$, $f_1,f_2:\br^2\times \br^m\to \br$ are smooth functions. 

In this case, the singular set
\[\Sigma=\{(x,y):g(x,y,\a)=0\}\]
is the zero curve of $g$. This curve is the boundary of two domains:
$$
\Sigma^+ = \{(x,y):g(x,y,\a)>0\}, \q 
\Sigma^- = \{(x,y):g(x,y,\a)<0\}.
$$
We assume that $g$ is such that $g_x$ has finitely many zeros on $\Sigma$. This means that $\nabla g$ is zero also at finitely many points of $\Sigma$. Every point of $\Sigma$ with $\nabla g \neq 0$ belongs to the closure of both $\Sigma^+$ and $\Sigma^-$.

In order to describe the dynamics of such a two-dimensional system, we introduce the basic dynamical elements, such as special points and cycles. 

\begin{Def}\label{Def:eq}\rm
A point $(x,y)$ is called
\begin{itemize}
\item an {\it equilibrium} if $f_1(x,y,\a)=f_2(x,y,\a)=0$; and $g(x,y,\a)\neq 0$; 
\item a {\it singular equilibrium} if $g(x,y,\a)=f_1(x,y,\a)=0$.
\item a {\it fold point} or a {\it fold}, if $g(x,y,\a)=g_x(x,y,\a)=0$; 
\end{itemize}
\end{Def}

\begin{rmk}\rm
The fold points given by Definition \ref{Def:eq} are also referred to as \textit{non-standard algebraic singular points} in DAE terminology (cf. \cite{RR_1994,RZ_2001}). The singular equilibria are also called as non-standard in some literature, e.g. in \cite{RR_1994}, or standard (as extended in \cite{RZ_2001}) geometric singular points.
\end{rmk}

Making time transformation $d\tau:=\frac{dt}{g}$ for $g\ne 0$ in (\ref{eq:dds_2D_g}), one obtains the {\it desingularized} system:
	\begin{equation}\label{eq:dds_2D_g_dis}
	\begin{cases}
	\dot x=f_1(x,y,\a)\\
	\dot y=f_2(x,y,\a)g(x,y,\a).
	\end{cases}
	\end{equation}
The correspondence between the systems is the following: every trajectory $\Gamma$ of system (\ref{eq:dds_2D_g_dis}) contains one or more trajectories of system (\ref{eq:dds_2D_g}). The intersection points of $\Gamma$ with singularity curve $g = 0$ (if any) split $\Gamma$ into connected pieces $\Gamma_1, \Gamma_2, \ldots$ that are  trajectories of (\ref{eq:dds_2D_g}) with the same direction of time as $\Gamma$ for every $\Gamma_i \subset \Sigma^+$ and with the opposite direction of time if $\Gamma_i \subset \Sigma^-$. It is clear that both equilibria and singular equilibria of system (\ref{eq:dds_2D_g}), given by Definition~\ref{Def:eq}, are equilibria of the ODE system (\ref{eq:dds_2D_g_dis}) in the usual sense.

\begin{rmk}\rm
Transforming the time by formula $d\tau:=\frac{dt}{-g}$ creates another desingularized system, in which the time flows in the opposite direction on every trajectory of (\ref{eq:dds_2D_g_dis}). 
The trajectories of system (\ref{eq:dds_2D_g}) follow the trajectories of this system in $\Sigma^-$ and flow in the opposite direction in $\Sigma^+$.
\end{rmk}

 \begin{figure}[!htb]
 	\centerline{
 		\includegraphics[width=1\textwidth]{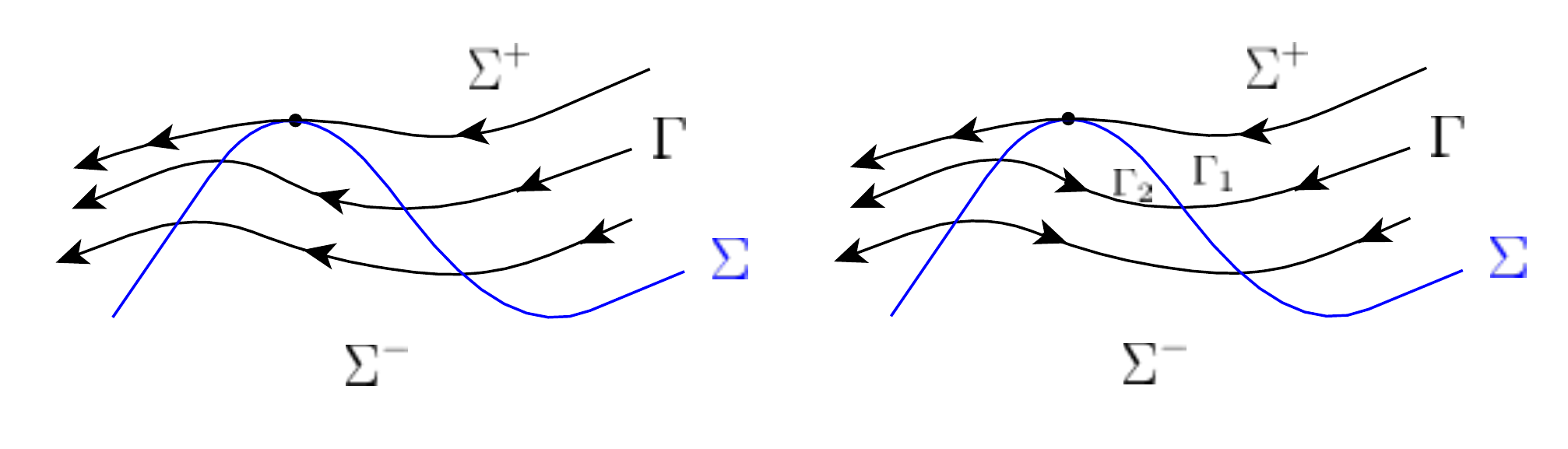}
 	} 
 	\caption{Left: the flow of desingularized ODE system (\ref{eq:dds_2D_g_dis}); Right: the flow of the original system. The part of $\Gamma$ lying in $\Sigma^-$ reverses time. The solid points indicate simple fold points where no change of flow direction is detected.}
 	\label{F:desing}
 \end{figure}

Consider  a point $M \in \Sigma$ and trajectory $\Gamma \ni M$ of ODE system (\ref{eq:dds_2D_g_dis}), see \cite{Rabier} for details. The trajectory will intersect $\Sigma$ transversely if it is not an equilibrium state and if its tangent vector $(f_1, f_2 \cdot g)^\top|_M = (f_1, 0)^\top$ is not orthogonal to $\nabla g = (g_x, g_y)^\top$, i.e. $f_1 \cdot g_x \neq 0$. Locally, $\Gamma$ is split by $M$ into two components, $\Gamma_1$ and $\Gamma_2$ such that $M = \Gamma(0)$, $\Gamma_1 \subset \Gamma(t)$ for $t < 0$ and $\Gamma_2 \subset \Gamma(t)$ for $t > 0$. 
If $f_1 \cdot g_x > 0$, then $\Gamma$ crosses $\Sigma$ from $\Sigma^-$ to $\Sigma^+$, $\Gamma_2 \subset \Sigma^+$ is the trajectory of (\ref{eq:dds_2D_g}), and $\Gamma_1 \subset \Sigma^-$ is the trajectory of (\ref{eq:dds_2D_g}) with time reversal. In this case we call point $M$ {\em outgoing}. When $f_1 \cdot g_x < 0$, the direction of time is preserved on $\Gamma_1 \subset \Sigma^+$ and reversed on $\Gamma_2 \subset \Sigma^-$, and point $M$ is called {\em incoming}. 
Inequalities $f_1 \cdot g_x > 0$ and $f_1 \cdot g_x < 0$ are open conditions, thus curve $\Sigma$ consists of incoming $\Sigma^{inc}$ and outgoing $\Sigma^{out}$ zones, separated by points where $f_1 \cdot g_x = 0$, those are either singular equilibria $f_1 = 0$ or fold points $g_x = 0$.


\begin{Def}\rm
A limit cycle of system (\ref{eq:dds_2D_g_dis}) is called a {\it limit cycle} of system (\ref{eq:dds_2D_g}), if it has no intersections with the singularity curve $\Sigma$. Otherwise, it is called a {\it folded limit cycle}.
\end{Def}

The limit cycle is a periodic orbit of system (\ref{eq:dds_2D_g}). A folded limit cycle consists of more than one  orbit of system (\ref{eq:dds_2D_g}).


\subsection{Structurally stable objects}\label{sec:str_stable}

\begin{Def}\rm
An equilibrium of system (\ref{eq:dds_2D_g}) is called {\it simple} or {\it hyperbolic}, if the linearization matrix of desingularized system (\ref{eq:dds_2D_g_dis}) in this point does not have eigenvalues on the imaginary axis.
\end{Def}

\begin{Def}\rm 
	A singular equilibrium $(x_0, y_0)$ of system (\ref{eq:dds_2D_g}) is called {\it simple}, if the following inequalities are fulfilled:
	\begin{equation}
	f_2(x_0, y_0, \a) \neq 0, \q g_x(x_0, y_0, \a) \neq 0, \q 
	\det \left.\frac{\partial (f_1,g)}{\partial(x,y)}\right|_{(x_0, y_0, \a)}\ne 0,
	\end{equation}
	and the linearization matrix of desingularized system (\ref{eq:dds_2D_g_dis}) in this point does not have eigenvalues on the imaginary axis.
\end{Def}

\begin{Def}\rm
	A fold $(x_0, y_0)$ of system (\ref{eq:dds_2D_g}) is called {\it simple}, if the following inequalities are fulfilled in it:
	\begin{equation}
	f_1(x_0, y_0, \a) \neq 0, \q g_y(x_0, y_0, \a) \neq 0, \q 
	g_{xx}(x_0, y_0, \a) \ne 0.
	\end{equation}
\end{Def}

\subsubsection{Simple Equilibria}\label{subsubsec:simple_eq}
Simple equilibria lie outside the singularity curve. A topological type of an equilibrium $M$ is determined by eigenvalues $\lambda_1$ and $\lambda_2$ of the linearization matrix
\begin{equation}\label{linear_eq}
  A_{EQ} = \left(
\begin{array}{cc}
f_{1x} & f_{1y}\\
gf_{2x} & gf_{2y}
\end{array}
\right).  
\end{equation}
For real $\lambda_1$ and $\lambda_2$, $M$ is saddle if $\lambda_1 \lambda_2 < 0$ and node if $\lambda_1 \lambda_2 > 0$. If the eigenvalues are a complex-conjugate pair, the equilibrium is a focus. A node or a focus $M$ is stable if
\begin{equation}
    M \in \Sigma^+, \Re \lambda_{1,2} < 0 \q \text{or} \q
    M \in \Sigma^-, \Re \lambda_{1,2} > 0, 
\end{equation}
and unstable {if
\begin{equation}
    M \in \Sigma^+, \Re \lambda_{1,2} > 0 \q \text{or} \q
    M \in \Sigma^-, \Re \lambda_{1,2} < 0. 
\end{equation}}

Simple equilibria persist under small perturbations, because they remain equilibria and retain their topological type in the desingularized system. Thus, in the original system, they lie outside the singularity curve $\Sigma$, and their topological type also does not change.

\subsubsection{Simple Singular Equilibria}\label{subsubsec:simple_seq}

In a similar way we classify simple singular equilibria using eigenvalues $\lambda_{1,2}$ of linearization matrix
\begin{equation}\label{linear_s_eq}
  A_{sEQ} = \left(
\begin{array}{cc}
f_{1x} & f_{1y}\\
g_xf_{2} & g_yf_{2}
\end{array}
\right).  
\end{equation}
%
\begin{Def}\rm Let $\lam_{1,2}$ be the eigenvalues of $A_{sEQ}$ in (\ref{linear_s_eq}) evaluated at a simple singular equilibrium $M$ of  (\ref{eq:dds_2D_g}). Then, $M$ is called a {\it folded node}, if $\lambda_{1,2} \in \mathbb{R}$ and $\lambda_1 \lambda_2 > 0$; a {\it folded saddle}, if $\lambda_{1,2} \in \mathbb{R}$ and $\lambda_1 \lambda_2 < 0$; and a {\it folded focus}, if $\lambda_{1,2} \notin \mathbb{R}$.
\end{Def}

The dynamics near a folded node and a folded saddle is determined by eigendirections corresponding to eigenvalues $\lambda_{1,2}$ and respective invariant manifolds. The following lemma states that these eigendirections are never tangent to $\Sigma$ at simple singular equilibria.

\begin{lem}\label{transversal}
In a folded node and a folded saddle, the eigendirections, corresponding to eigenvalues $\lambda_{1,2}$ are transverse to the singularity curve $\Sigma$.
\end{lem}

\begin{proof}
We prove the lemma by contradiction. Assume that for some eigenvalue $\lambda_1$ its eigenvector is tangent to $\Sigma$. Then, tangent to $\Sigma$ vector $(g_y, -g_x)^\top$ is the eigenvector of matrix $A_{sEQ}$:
$$
\left(\begin{array}{cc}
     f_{1x} - \lambda_1 & f_{1y}  \\
     g_x f_2 & g_y f_2 - \lambda_1 
\end{array} \right)
\left(\begin{array}{c}
     g_y  \\
     -g_x 
\end{array} \right) = 
\left(\begin{array}{c}
     f_{1x}g_y - f_{1y} g_x - \lambda_1 g_{y}  \\
     g_x \lambda_1
\end{array} \right) = 0,
$$
which implies either $g_x = 0$ or $\lambda_1 = 0$, both conditions contradict the assumption that the considered singular equilibrium is simple.
\end{proof}

To describe the dynamical properties of all three types of singular equilibria, we consider a small neighbourhood $U$ of $M$. Locally, $U$ is divided by curve $\Sigma$ into two disconnected parts {given by} $\Sigma^+_{loc} \subset \Sigma^+$ and $\Sigma^-_{loc} \subset \Sigma^-$. 

\textbf{Folded node.} Consider a folded node $M$ with $\lambda_2 < \lambda_1 < 0$ (for the case $0 < \lambda_1 < \lambda_2$ the statement will be the same with $\Sigma^+$ and $\Sigma^-$ interchanged). It has a leading direction $e^L$ defined by the eigenvector corresponding to $\lambda_1$ and a non-leading direction $e^{nL}$ defined by the eigenvector corresponding to $\lambda_2$. There exists a semi-stable smooth invariant manifold $W^{nL}(M)$ tangent to $e^{nL}$ at point $M$. Its existence follows from the desingularized system (\ref{eq:dds_2D_g_dis}), this system has at $M$ a stable or a completely unstable node equilibrium, that possesses a smooth strong stable (unstable) manifold. After coming back to the original system, a part of this manifold lying in $\Sigma^-$ changes the direction of time, so that the manifold becomes a semi-stable manifold $W^{nL}(M)$.

Manifold $W^{nL}(M)$ and $\Sigma$ intersect transversely according to Lemma~\ref{transversal}. They divide neighbourhood $U$ into four sectors, we will call them incoming, stable, outgoing and unstable.
The \textbf{incoming} sector lies in $\Sigma^+_{loc}$ every initial condition from this sector reaches $\Sigma^{inc}$ in forward time and leave $U$ in backward time. The \textbf{stable} sector also lies in $\Sigma^+_{loc}$ and contains the leading direction $e^L$. The trajectories from this sector reach $M$ tangent to $e^L$ in  forward time and reach $\Sigma^{out}$ in  backward time. 

In a similar way, we describe the dynamics in $\Sigma^-_{loc}$: it is divided by $W^{nL}(M)$ into unstable and outgoing sectors. In the \textbf{outgoing} sector trajectories leave $U$ in forward time and reach $\Sigma^{out}$ in backward time. In the \textbf{unstable} sector the trajectories reach $\Sigma^{inc}$ in forward time and point $M$ in backward time. All four types of behavior are illustrated at Figure \ref{F:folded_node}.

 \begin{figure}[!htb]
	\centerline{
		\includegraphics[width=.4\textwidth]{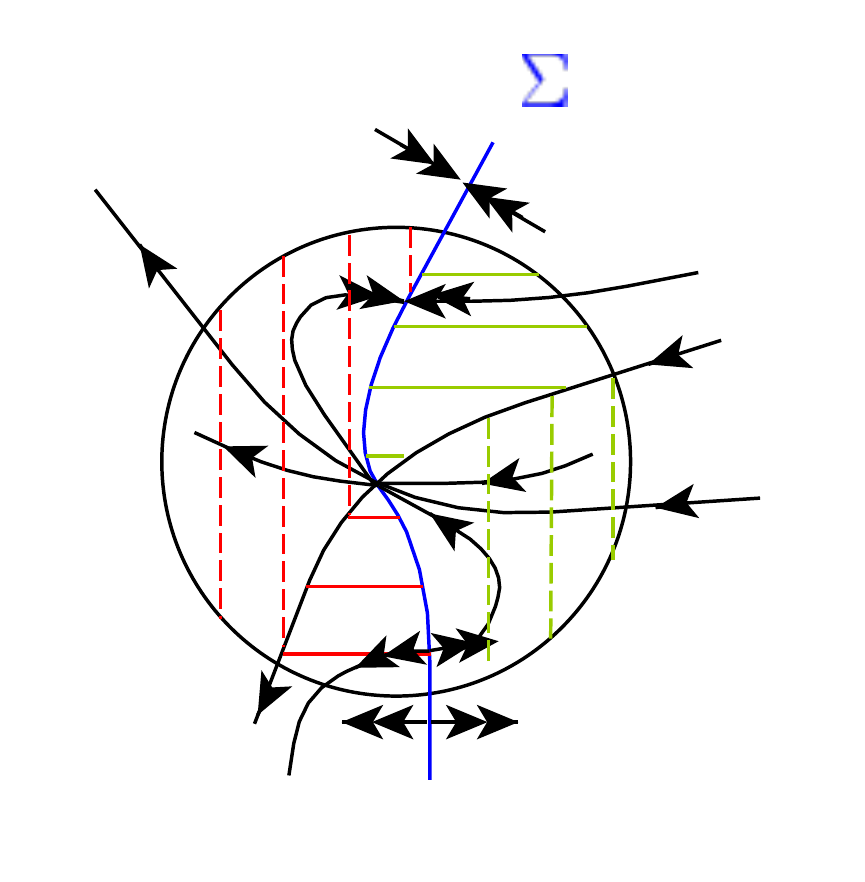}
	} 
	\caption{Dynamics around a simple singular equilibrium: the case of a folded node. The solid green and red lines mark the incoming and outgoing sectors, respectively. The dashed green and red  lines mark the stable and unstable sectors, respectively.}
	\label{F:folded_node}
\end{figure}

\textbf{Folded saddle.} In a folded saddle $M$ the eigenvalues of the desingularized linearization matrix are $\lambda_1 < 0 < \lambda_2$. Point $M$ belongs to two smooth invariant manifolds: $W^-$ tangent to eigendirection $e^-$ corresponding to $\lambda_1$ and $W^+$ tangent to eigendirection $e^+$ corresponding to $\lambda_2$. They divide  $\Sigma^+$ into three sectors: incoming, saddle and outgoing. The incoming sector is bounded by $\Sigma^{inc}$ and $W^-$, all orbits from it reach $\Sigma_{inc}$ in forward time and leave $U$ in backward time. The saddle sector is bounded by $W^+$ and $W^-$, all orbits leave $U$ in both directions of time. The outgoing sector is bounded by $\Sigma^{out}$ and $W^+$, the orbits in it leave $U$ in forward time and reach $\Sigma^{out}$ in backward time. All above describes the dynamics also in $\Sigma^-$, where the outgoing sector is bounded by $\Sigma^{out}$ and $W^-$ and the incoming sector is bounded by $\Sigma^{inc}$ and $W^+$. The stable manifold of $M$ is $W^s(M) = M \cup (W^- \cap \Sigma^+) \cup (W^+ \cap \Sigma^-)$, the unstable is $W^u(M) = M \cup (W^+ \cap \Sigma^+) \cup (W^- \cap \Sigma^-)$, both manifolds are $C^0$ in $M$. See Figure \ref{F:folded_saddle}.

 \begin{figure}[!htb]
	\centerline{
		\includegraphics[width=.4\textwidth]{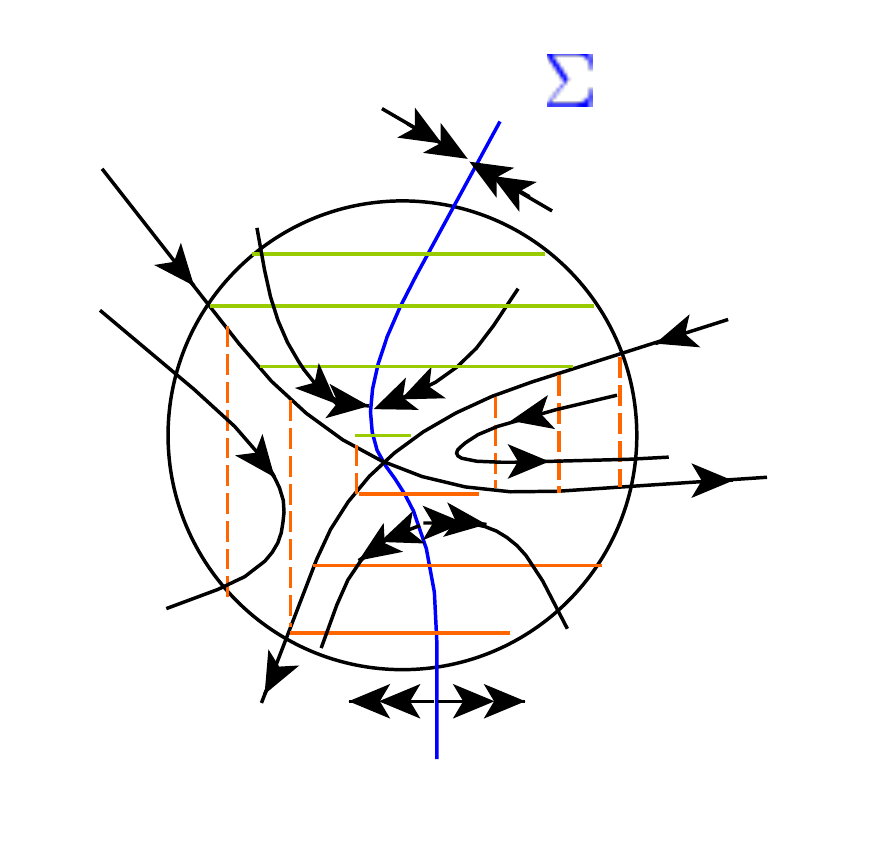}
	} 
	\caption{Dynamics around a simple singular equilibrium: the case of a folded saddle. The solid green and red lines mark the incoming and outgoing sectors, respectively. The dashed red  lines mark the unstable sector.}
	\label{F:folded_saddle}
\end{figure}

\textbf{Folded focus.} Near a folded focus all orbits reach $\Sigma^{out}$ in backward time and $\Sigma^{inc}$ in forward time. See Figure \ref{F:folded_focus}.

 \begin{figure}[!htb]
	\centerline{
		\includegraphics[width=.3\textwidth]{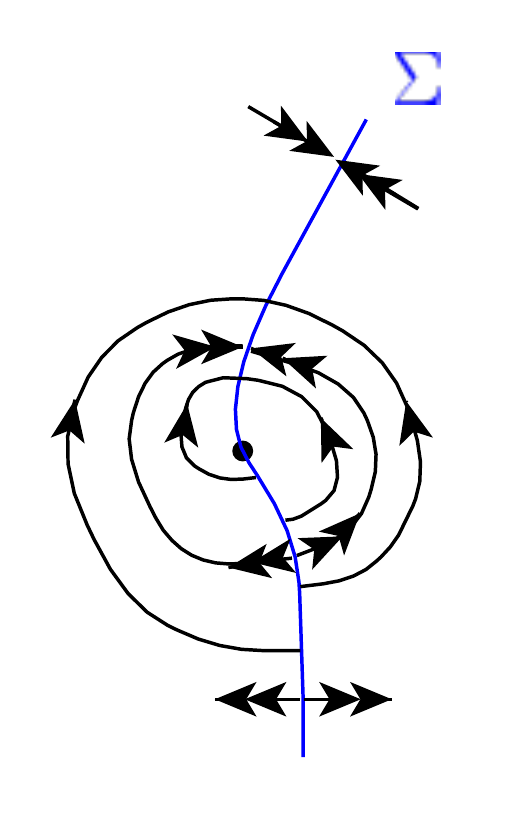}
	} 
	\caption{Dynamics around a simple singular equilibrium: the case of a folded focus.}
	\label{F:folded_focus}
\end{figure}

Under small (smooth) perturbations, simple singular equilibria persist and retain their topological type. The reasons of it are that matrix (\ref{linear_s_eq}) is non-degenerate at such point, and by the Implicit Function Theorem equation $g = 0, f_1 = 0$ has a unique solution for small $\alpha$, provided that it exists for $\alpha = 0$. Also, as the topological type of such an equilibrium persists in the desingularized system, it also persists in the original one.

\subsubsection{Simple Fold}\label{subsubsec:simple_fp}

Near the simple fold point $M(x_0, y_0)$, where $(g = g_x = 0, \; g_{xx}\neq 0, \; g_y \neq 0)$, equation $g(x, y) = 0$ of the singularity curve $\Sigma$ can be locally explicitly resolved as a function 
$\displaystyle y = \psi(x) = - \frac{g_{xx}(M)}{g_y(M)} x^2 + O(x^3)$, $y_0 = \psi(x_0)$. The fold point is a local maximum or minimum of this function.
Consider the desingularized system (\ref{eq:dds_2D_g_dis}) and its solution with initial condition $(x_0, y_0)$. At the point $M$ its $y$ component $f_2(x_0, y_0) g(x_0, y_0)$ vanishes, thus the trajectory is tangent to $\Sigma$, see Fig.~\ref{F:desing}. 

The simple fold is either $\Sigma^+$-convex, when 
$g_{xx}(M) g_y(M) > 0$ and $\Sigma^-$-convex, when 
$g_{xx}(M) g_y(M) < 0$.
Also, it persists under small perturbations. Indeed, for system of equations $g = g_x = 0$ we have
\begin{equation}\label{IFT_Fold}
    \displaystyle \frac{\partial (g, g_x)}{\partial (x, y)}(x_0, y_0) = 
    \left(\begin{array}{cc}
     0 & g_{y}  \\
     g_{xx} & g_{xy}\end{array} \right) (x_0, y_0) = 
     -g_{xx}(x_0. y_0) g_y(x_0, y_0) \neq 0
\end{equation}
so that by the Implicit Function Theorem the system can be uniquely solved with respect to $(x, y)$ for all small $\alpha$. This solution gives a unique fold point in a small neighbourhood of point $M$. In addition, condition $f_1(x, y) \neq 0$ is fulfilled in some small neighbourhood of the fold point, also for small $\a$, thus no other objects (regular or singular equilibria) appear there under small perturbations.

\subsubsection{Regular and folded limit cycles}

The limit cycles, regular and folded, both correspond to limit cycles of the desingularized system (\ref{eq:dds_2D_g_dis}). By standard methods (the Poincar\'e crossection) one defines the multiplier $\mu > 0$ of such an orbit. A regular cycle is simple (structurally stable), if its multiplier differs from one. A folded cycle is structurally stable also if $\mu \neq 1$, and, in addition, it intersects the singularity curve only transversely. The possible types of simple limit cycles are illustrated in Fig.~\ref{F:2d_cycle}.

\begin{figure}[!htb]
	\includegraphics[angle=90,origin=c,width=.32\textwidth]{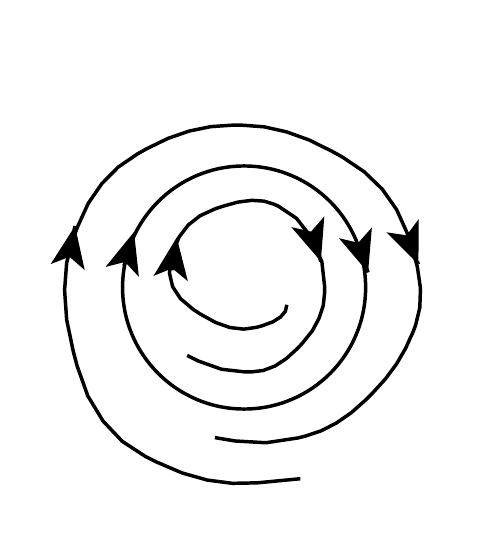}
	\vskip-.1cm\hskip.5cm
	\includegraphics[angle=90,origin=c,width=.3\textwidth]{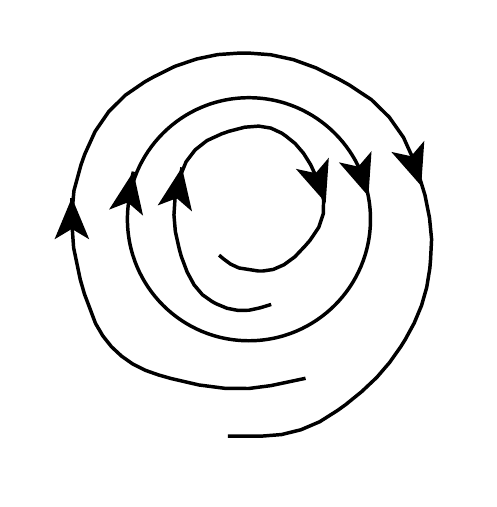}		
	\vskip-4.5cm\hskip6cm
\includegraphics[angle=90,origin=c,width=.36\textwidth]{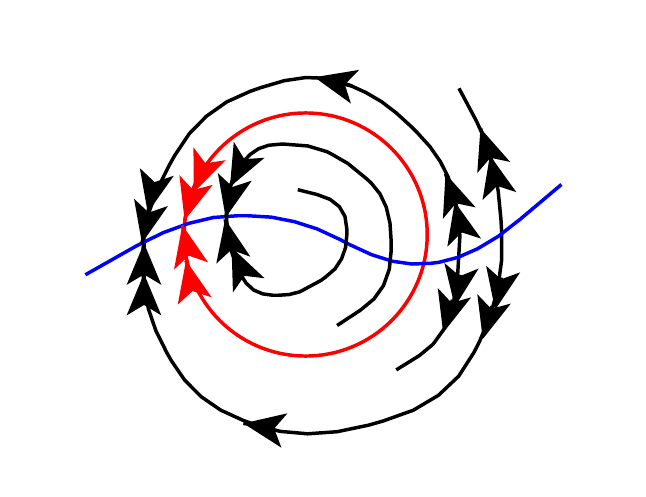}			\vskip-3.9cm\hskip9.8cm
	\includegraphics[width=.1\textwidth]{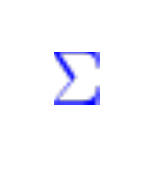}	
	\vskip3cm
	\caption{\label{F:2d_cycle} Dynamics around a stable limit cycle (top left), an unstable limit cycle (bottom left) and a folded limit cycle (right) of (\ref{eq:dds_2D_g}). }
	
\end{figure}


\subsection{Bifurcations}

The bifurcations in two-dimensional DAEs are divided in three main groups: geometric (bifurcations of the singularity set), local and global bifurcations. 

\subsubsection{Geometric bifurcations}

Geometric bifurcations are related to the reconstruction of the topology of the singularity set $\Sigma$. It means that after arbitrary small perturbations the set $\Sigma$ is topologically not equivalent to itself at the initial parameter value.
This happens when its branches appear, disappear or interact with each other.
Among codimension-one bifurcations, there are those, related to failure of local existence of a unique branch of $\Sigma$, i.e. existence of a point $(x, y)$, where $\nabla g(x,y) = 0$. At the same time, the Hessian should be non-zero at the bifurcation moment, so that the codimension is not higher than one: 
\begin{equation}\label{top_bif}
   \nabla g(x,y) = 0, \;\; \det D^2 g(x, y) \neq 0 
\end{equation}

Depending on the sign of the Hessian, two cases are possible \cite{Golubitsky_1985}:

  \textbf {T1. Hyperbolic bifurcation.} $g_x = g_y = 0$ and $D^2 g(x,y) < 0$. For example, $g(x,y,\a)=x^2-y^2-\a$ at $(0,0,0)$. See Figure \ref{F:2d_topo_hyp}.
    \begin{figure}[!htb]
\centerline{
 		\includegraphics[width=.8\textwidth]{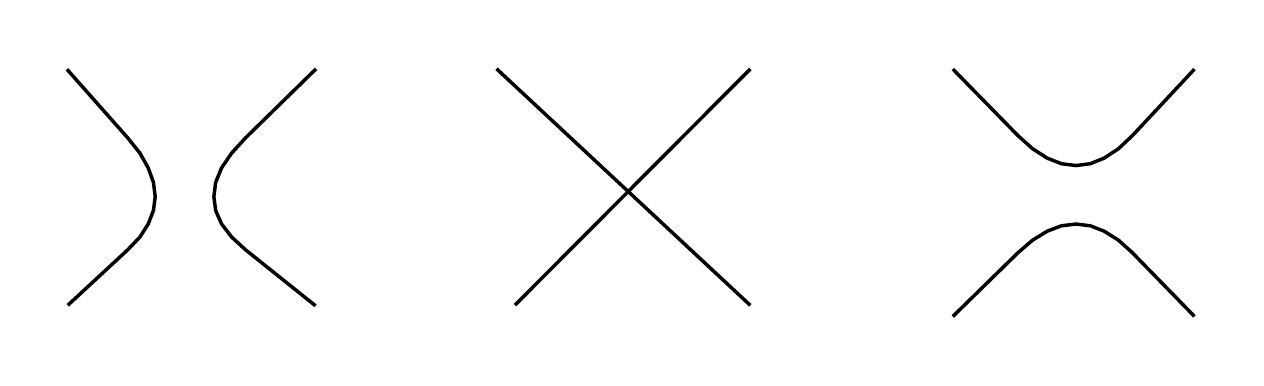}
 	} 
 	\caption{Hyperbolic singularity curve bifurcation of (\ref{eq:dds_2D_g}). }
	\label{F:2d_topo_hyp}
 \end{figure}

\textbf{T2. Elliptic bifurcation.} $g_x = g_y = 0$ and $D^2 g(x,y) > 0$. For example, $g(x,y,\a)=x^2+y^2-\a$. See Figure \ref{F:2d_topo_elip}.

\begin{figure}[!htb]
\centerline{
 		\includegraphics[width=.7\textwidth]{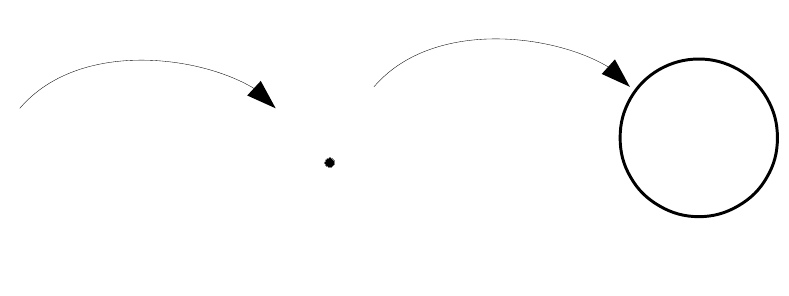}
 	} 
 	\caption{Elliptic singularity curve bifurcation of (\ref{eq:dds_2D_g}). }
 	\vskip-.5cm
	\label{F:2d_topo_elip}
 \end{figure}
 

\subsubsection{Local bifurcations}\label{sec:loc_bif}

Local bifurcations occur to simple objects described in subsection~\ref{sec:str_stable}, when they lose their simple properties. From the list below, bifurcations \textbf{L1, L6, L8} occur outside the singularity curve $\Sigma$ and thus they are unfolded as in regular ODEs. Bifurcations \textbf{L2, L7, L9} involve the curve $\Sigma$, but after the desingularization procedure, they become bifurcations \textbf{L1, L6, L8} respectively, in the ODE system (\ref{eq:dds_2D_g_dis}) and their unfolding can be described accordingly. The rest, bifurcations \textbf{L3--L5}, either are unfolded in a different way in regular systems or do not have regular analogues at all. The latter are studied in detail in Section~\ref{sec:loc_norm}.


\textbf{L1. Saddle-node.} 
      This bifurcation occurs when at the equilibrium point $M$ the eigenvalues of the linearization matrix (\ref{linear_eq}) are $\lambda_1 = 0$ and $\lambda_2 \neq 0$, i.e. when $f_{1x}f_{2y} - f_{1y}f_{2x} = 0$. Like in a regular ODE system, a codimension-one saddle-node under small perturbations either disappears so that in some small neighbourhood there are no equilibria, or splits into two simple equilibria, a saddle and a node.
      
\textbf{L2. Singular saddle-node of type I.} (according to the classification by  C. Kuehn~\cite{CKuehn}). This bifurcation corresponds to the existence of a singular equilibrium, for which linearization matrix 
      (\ref{linear_s_eq}) has eigenvalues $\lambda_1 = 0$ and $\lambda_2 \neq 0$. This happens when $f_{1x}g_y - f_{1y}g_x = 0$. A codimension-one singular equilibrium of type I under small perturbations either disappears or splits into two simple singular equilibria -- a folded saddle and a folded node, see Fig.~\ref{F:vEQ_vEQ}. 
      
\begin{figure}[!htb]
	\centerline{
		\includegraphics[width=.8\textwidth]{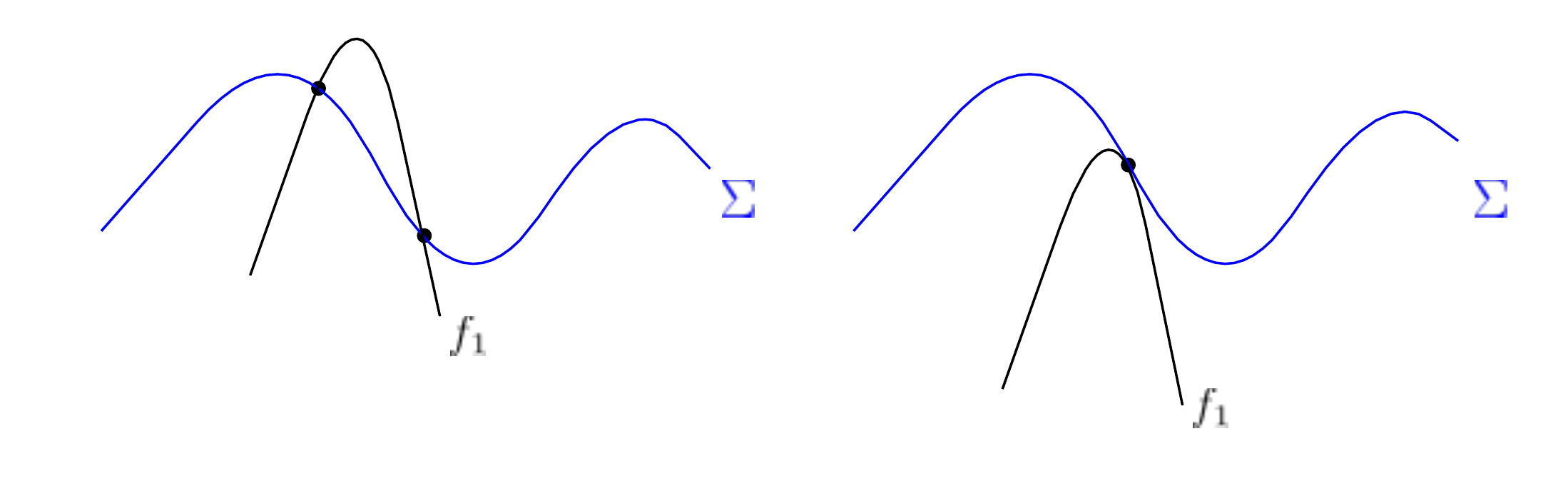}
	} 
	\caption{Singular saddle-node type 1: $g = f_1 = 0$, $f_{1x}g_y - f_{1y}g_x = 0$, $g_{x}\ne 0$.}
	\label{F:vEQ_vEQ}
\end{figure}

    \textbf{L3. Singular saddle-node of type II.} (according to the classification by  C. Kuehn~\cite{CKuehn}). 
      This bifurcation occurs when a  singular equilibrium point $M$ ($f_1 = g = 0$) also satisfies regular equilibrium condition $f_2 = 0$. A small perturbation of a codimension-one singular saddle-node of type II 
      leads to the appearance of a simple equilibrium  and a simple singular equilibrium. They appear in combinations either node + folded saddle, or saddle + folded node. For the derivation of the normal form refer to Lemma \ref{lem:eq_seq_bf} below. The bifurcation is illustrated on Fig.~\ref{F:2d_++}, see also \cite{Venk_1992, Venk_1995, B_1998, B_2000, B_2001, Riaza_2002}, where such a bifurcation was studied.
      
\begin{figure}[!htb]
		\includegraphics[width=1\textwidth]{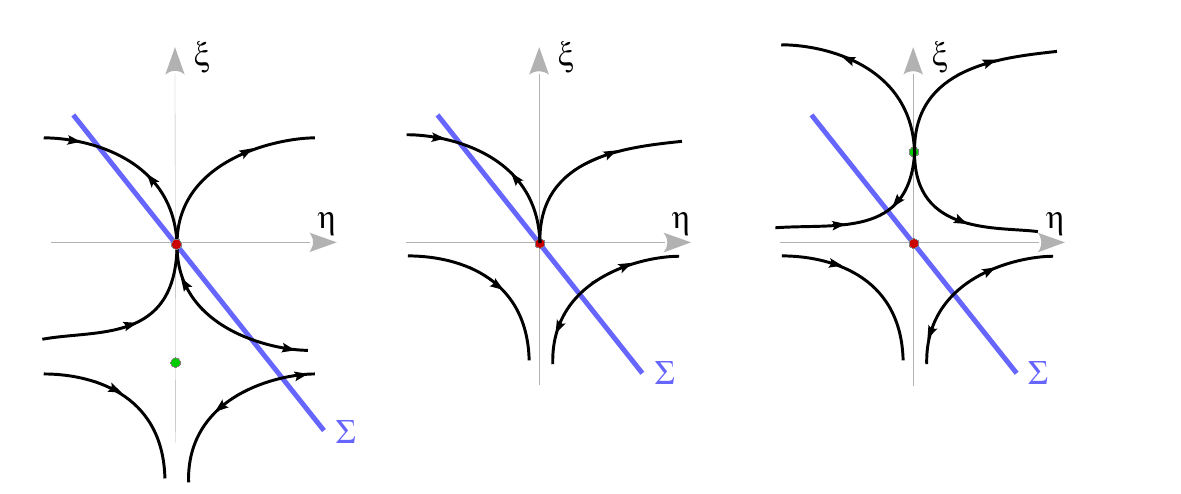}
	\caption{The singular saddle-node bifurcation of type II at $(\xi^*,\eta^*,\a^*)=(0,0,0)$ 
	for  $\Del_4>0$ and $f_{1x}>0$. As $\a$ changes from negative to positive, the local flow changes from the left figure to right figure.} 
	\label{F:2d_++}
\end{figure}
      
\textbf{L4. Cubic fold.} A fold point $g = g_x = 0$, $g_y \neq 0$, is non-simple of codimension one (a cubic fold) if $g_{xx} = 0$ and $g_{xxx} \neq 0$. Under small perturbations the cubic fold generically either disappears or splits into a pair of simple folds with the opposite convexity, as stated by Lemma~\ref{lem:cubic_fold}. The bifurcation is illustrated in Fig.~\ref{F:fold_fold}.
     
\begin{figure}[!htb]
 	\centerline{
 		\includegraphics[width=.8\textwidth]{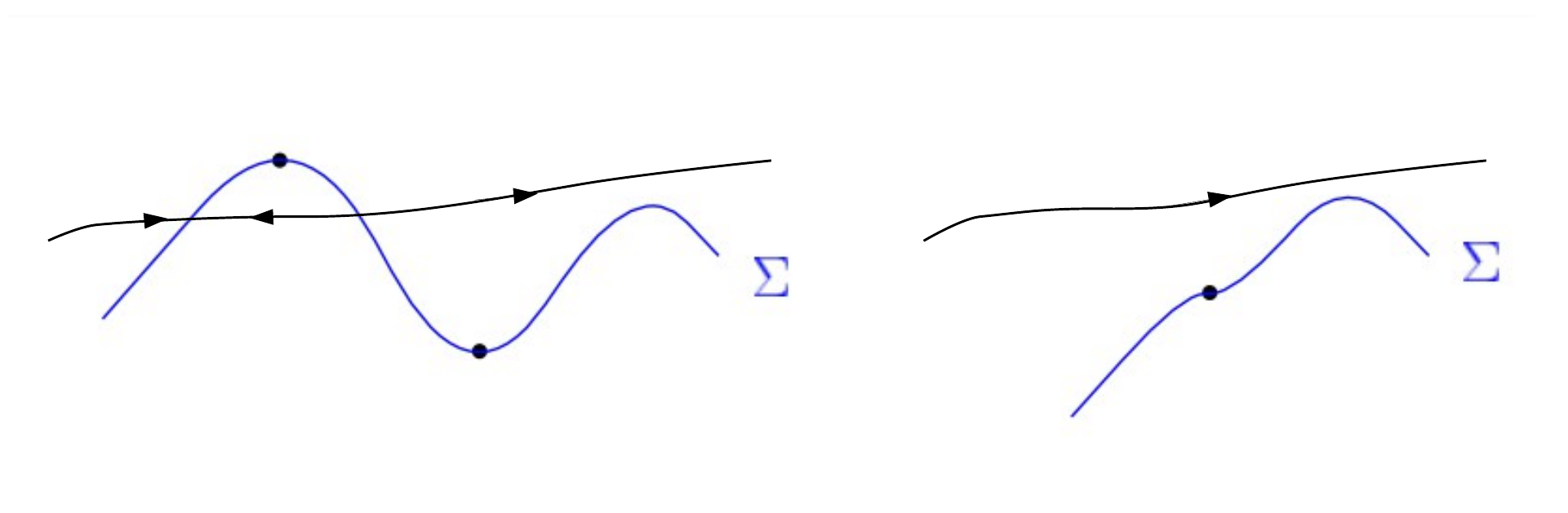}
 	} 
 	\caption{A cubic fold bifurcation $g=g_x=g_{xx}=0$, $f_1\ne 0$.}
 	\label{F:fold_fold}
 \end{figure}
      
\textbf{L5. Singular equilibrium-fold} This bifurcation occurs when at a point $M$ both conditions of a fold and a singular equilibrium are fulfilled, i.e. $g = g_x = f_1 = 0$ and $f_2 \neq 0$, $f_{1x} \neq 0$, $g_{xx} \neq 0$, $g_y \neq 0$. Under small perturbations, a singular equilibrium-fold splits into a simple fold and a simple singular equilibrium, see Lemma~\ref{L:fold_sEQ} for details. The bifurcation is illustrated at  Fig.~\ref{F:fold_vEQ}.

\begin{figure}[!htb]
	\centerline{
		\includegraphics[width=\textwidth]{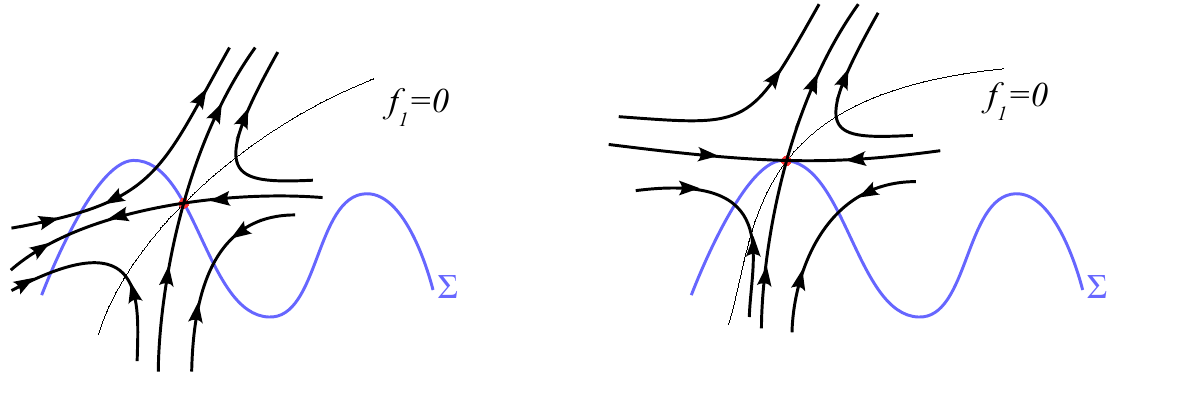}
	} 
	\caption{A singular equilibrium-fold bifurcation $g=g_x=f_1=0$, $g_{xx}\ne 0$.}
	\label{F:fold_vEQ}
\end{figure}

\textbf{L6. Transition between folded node and folded focus.} This bifurcation occurs at a singular equilibrium, when the two eigenvalues of the linearization matrix  (\ref{linear_s_eq}) coincide. In small perturbations they either become a real (folded node) or a complex-conjugated (folded focus) pair of different eigenvalues. Note that the similar transition for a regular equilibrium is not considered a bifurcation -- the local flows around a node and a focus can be topologically conjugated. However, in the case of a singular equilibrium, the local flows near a folded node and a folded focus are not similar: near the folded node there exists a subset of points such that their trajectories reach the singular equilibrium in forward or backward time, while in the neighborhood of a folded focus there are no such orbits (see Subsection~\ref{subsubsec:simple_seq} for details).  
 
\textbf{L7. Andronov-Hopf bifurcation.} This standard bifurcation occurs when a stable (unstable) focus equilibrium has a pair of pure imaginary eigenvalues (of matrix (\ref{linear_eq})). Under small perturbations such a weak focus either becomes a simple stable (unstable) focus, or unstable (stable) focus and a stable (unstable) limit cycle is born.

\textbf{L8. Folded Andronov-Hopf bifurcation.} This bifurcartion takes place, when linearisation matrix (\ref{linear_s_eq}) has a pair of pure imaginary eigenvalues. This bifurcation corresponds to a regular Andronov-Hopf in the desingularized system (\ref{eq:dds_2D_g_dis}), in which a limit cycle is born from a focus equilibrium. In the original system (\ref{eq:dds_2D_g}) under small perturbations a folded limit cycle is born from a folded focus.

\textbf{L9. Double limit cycle.} Existence of a limit cycle with multiplier equal to $+1$. Under small perturbations this cycle either disappears or is split into stable and unstable simple limit cycles.
          
\textbf{L10. Folded double limit cycle.} This bifurcation corresponds to the existence of a double limit cycle in the desingularized system (\ref{eq:dds_2D_g_dis}), that intersects the singularity curve. Under small perturbations  such a cycle either disappears or is split into a pair of simple folded limit cycles.

\begin{figure}[!htb]
		\includegraphics[angle=90,origin=c,width=.4\textwidth]{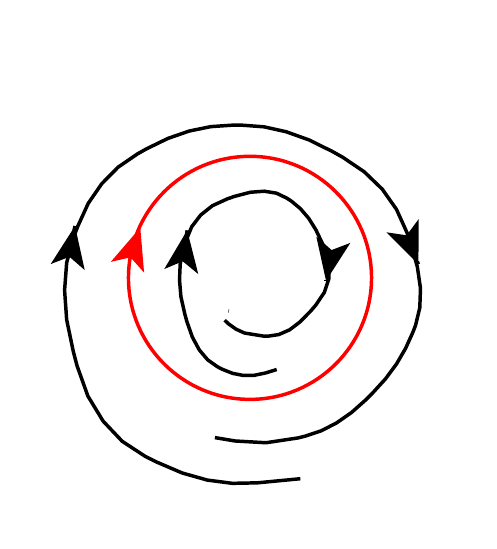}
		\vskip-5cm\hskip6cm
		\includegraphics[angle=90,origin=c,width=.46\textwidth]{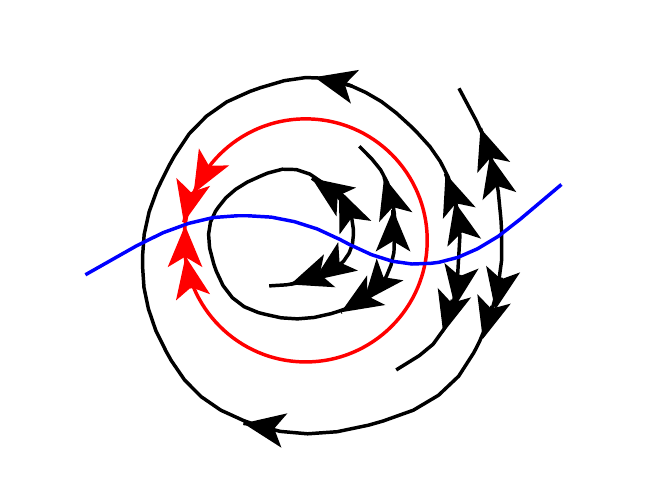}		
		\vskip-3.9cm\hskip10.8cm
\includegraphics[width=.1\textwidth]{2d_semi_stablec.pdf}	
	\vskip2cm
	\caption{Dynamics around a double limit cycle (left) and a folded double limit cycle (right) of (\ref{eq:dds_2D_g}). }

	\label{F:2d_semi_stable}
\end{figure}
      
    
 \subsubsection{Global Bifurcations}    
 This is the class of bifurcations, in which special orbits (homoclinic or heteroclinic) exist in the system. They are usually destroyed by small perturbations. Such orbits can be regular, when they have no intersections with the singularity curve $\Sigma$, or folded if such an intersection point exist.
    
\textbf{G1. A homoclinic orbit to a saddle.} In the desingularized system (\ref{eq:dds_2D_g_dis}) there exists a homoclinic loop $\Gamma$ to a saddle equilibrium, that is also an equilibrium in system (\ref{eq:dds_2D_g}).
        
        \textbf{G1a. Regular.} This is the standard homoclinic bifurcation, when the image of $\Gamma$ does not intersect $\Sigma$. Upon small perturbations it either just disappears, or disappears with the creation of a limit cycle.
        
        \textbf{G1b. Folded.} The image of $\Gamma$ in (\ref{eq:dds_2D_g}) has intersections with $\Sigma$. Under small pertubations, such a homoclinic loop gives rise to a folded limit cycle.
    
\textbf{G2. A homoclinic orbit to a folded saddle} a homoclinic orbit $\Gamma$ exists in the desingularized system (\ref{eq:dds_2D_g_dis}). In the original system (\ref{eq:dds_2D_g}) this equilibrium lies at $\Sigma$
  
  \textbf{G2a. Regular.}
   In the original system (\ref{eq:dds_2D_g}) the image of $\Gamma$ does not intersect $\Sigma$. Under small perturbations, when the folded saddle disappears, a regular limit cycle is born.
  
\textbf{G2b. Folded.} 
   In the original system (\ref{eq:dds_2D_g}) the image of $\Gamma$  has intersections with $\Sigma$. Under small perturbations, when the folded saddle disappears, a folded limit cycle is born.
    
\textbf{G3.  A homoclinic orbit to a fold point.} In the desingularized system (\ref{eq:dds_2D_g_dis}) there exists a simple limit cycle $L$. In the original system (\ref{eq:dds_2D_g}) the image of curve $L$ is tangent to $\Sigma$ at a simple fold point $F$.
     
\textbf{G3a. Regular.}
       The image of curve $L$ does not have common points with $\Sigma$ other than $F$.
       Under small perturbations, cycle $L$ in system (\ref{eq:dds_2D_g_dis}) persists, and in system (\ref{eq:dds_2D_g}) this curve becomes either a regular or a folded limit cycle.
     
\textbf{G3b. Folded.}
The image of curve $L$ have intersections with $\Sigma$ other than $F$.
     Under small perturbations,  cycle $L$ in system (\ref{eq:dds_2D_g_dis}) persists, and in system (\ref{eq:dds_2D_g}) it becomes a folded limit cycle.
        
\textbf{G4. A homoclinic orbit to a saddle-node.} In the desingularized system (\ref{eq:dds_2D_g_dis}) there exists a homoclinic orbit $L$ to a saddle-node equilibrium $M$. In system (\ref{eq:dds_2D_g}) point $M$ does not belong to singularity curve $\Sigma$. Upon the disappearance of the equilibrium a limit cycle is born in system (\ref{eq:dds_2D_g_dis}). 

         \textbf{G4a. Regular.} The image of $L$ in system (\ref{eq:dds_2D_g}) does not intersect the singularity curve $\Sigma$. Upon the disappearance of the equilibrium a limit cycle  is born also in the original system (\ref{eq:dds_2D_g}).
         
         \textbf{G4b. Folded.}  The image of $L$ in system (\ref{eq:dds_2D_g}) intersects transversely the singularity curve $\Sigma$. Upon the disappearance of the equilibrium a folded limit cycle is born.
     
\textbf{G5. A homoclinic orbit to  a singular saddle-node of type I.} In the desingularized system (\ref{eq:dds_2D_g_dis}) there exists a homoclinic orbit $L$ to a saddle-node equilibrium $M$. In system (\ref{eq:dds_2D_g}) point $M$ is a singular equlibrium (a singular saddle-node of type I). Upon the disappearance of the equilibrium a periodic orbit is born in system (\ref{eq:dds_2D_g_dis}). 

         \textbf{G5a. Regular.} The image of $L$ in system (\ref{eq:dds_2D_g}) does not intersect the singularity curve $\Sigma$. Upon the disappearance of the singular equilibrium a limit cycle is born in the original system (\ref{eq:dds_2D_g}).
         
         \textbf{G5b. Folded.}  The image of $L$ in system (\ref{eq:dds_2D_g}) intersects transversely the singularity curve $\Sigma$. Upon the disappearance of the singular equilibrium a folded limit cycle is born.
     
     
\textbf{G6. A heteroclinic connection.} 
This bifurcation corresponds to the existence of such an orbit $L$ in system (\ref{eq:dds_2D_g_dis}) that
      \begin{itemize}
          \item passes through two different fold points of system (\ref{eq:dds_2D_g}). In system (\ref{eq:dds_2D_g}) a piece of $L$ that lies between the fold points, is a heteroclinic connection of two folds (see Fig.~\ref{F:2d_heteroclinic_reg}).
          
          \item passes through one fold point $F$ of system (\ref{eq:dds_2D_g}) and tends to an equilibrium in forward or backward time, without passing through other fold points that differ from $F$. This is a heteroclinic connection of a fold and a regular or singular equilibrium in system (\ref{eq:dds_2D_g})
          
          \item tends to different equilibria in both forward and backward time, without passing through any of folds of system (\ref{eq:dds_2D_g}). This is a heteroclinic connection between two equilibria, two  singular equilibria or between an equilibrium and a singular equilibrium.
      \end{itemize}

All the heteroclinic connections listed above can be either regular, when they do not intersect the singularity curve, or folded, when such an intersection exists. Under small perturbations heteroclinic connections are generically broken.

         
         \begin{figure}[!htb]
         	\centerline{
         		\includegraphics[width=.8\textwidth]{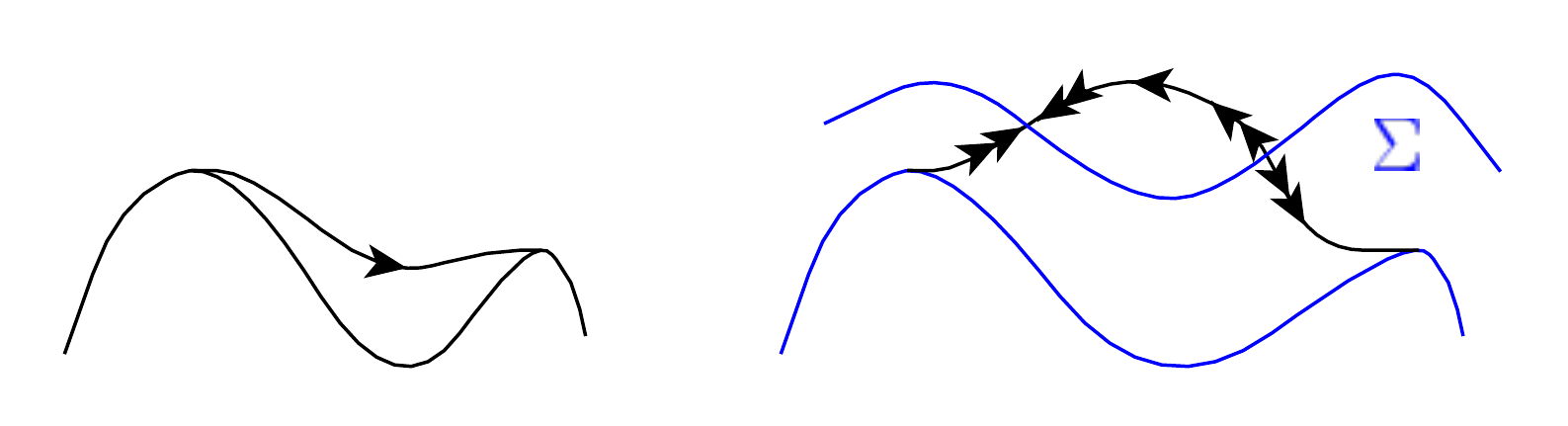}
         	}  
         	\vskip-.2cm
         	\caption{\label{F:2d_heteroclinic_reg} Dynamics around a regular heteroclinic connection (left) and a folded heteroclinic connection (right) of (\ref{eq:dds_2D_g}).}

         \end{figure}


\section{Local bifurcations L3--L5}\label{sec:loc_norm}

In this section the character of those local bifurcations, that do not have analogues in regular ODEs is studied in detail. These bifurcations are \textbf{L3:} Singular saddle-node of type II, \textbf{L4:} Cubic fold and \textbf{L5:} Singular equilibrium-fold from subsection~\ref{sec:loc_bif}. 

We introduce the following notations: 
%
%
\begin{equation}\label{eq:determinants}
\begin{array}{c}
\displaystyle
\Del_1 = \det \left.\frac{\partial (f_1,f_2)}{\partial(x,y)}\right|_{(0,0,0)},\q
%
\Del_2 = \det \left.\frac{\partial (f_1, g)}{\partial(x,y)}\right|_{(0,0,0)},\q 
\Del_3 = \det \left.\frac{\partial 
(g, g_x)}{\partial(y , \a)}
\right|_{(0,0,0)},
\\
%
%
%
\displaystyle
\Del_4 = 
\left.\frac{\partial (f_1,f_2,g)}{\partial(x,y,\a)}\right|_{(0,0,0)} ,\q
%
\Del_5=
\left.\frac{\partial (f_1, g, g_x)}{\partial(x,y,\a)}\right|_{(0,0,0)} .\q
%
	\end{array}
\end{equation}






\subsection{L3: Singular saddle-node of type II} 
This bifurcation occurs at point $M(0, 0)$ when both conditions for an equilibrium and a singular equilibrium are fulfilled in it for $\a = 0$. This means that $g(0,0,0) = f_1(0, 0, 0) = f_2(0, 0, 0) = 0$. We assume that the 
following inequalities also hold, so that the codimension is equal to one and that the parametric family is transversal:
\begin{equation}
\label{noSingFD}
\begin{array}{c}
   
\displaystyle    
f_{1x}(0,0,0)\ne 0, \q g_x(0, 0, 0) \neq 0, \q

\Del_1 \neq 0, \q \Del_2 \neq 0, \q \Del_4 \neq 0.
%
\end{array}
\end{equation}


\begin{lem}\label{lem:eq_seq_bf}
{Assume that genericity (codimension one + transversality) conditions (\ref{noSingFD}) are fulfilled at a singular saddle-node type II point $M$. Then in a generic unfolding point $M$ splits into a pair of structurally stable points, a regular and a singular equilibria. They are either a saddle and a folded node or a node and a folded saddle. 
%
%
}
\end{lem}

\begin{proof}

First of all, we note that equations $f_1 = f_2 = g = 0$ have a solution $x = y = \a = 0$, and by the transversality condition $\Del_4 \neq 0$
no other solution exists nearby, thus for small $\a \neq 0$ no other singular saddle-nodes of type II exist in some small neighbourhood of the origin.

Condition $\Del_1 \ne 0$ implies that for all small $\a$ equations $f_1 = f_2 = 0$ have a unique equilibrium solution
\begin{equation}\label{eq:eq}
\displaystyle  (x_e(\a), y_e(\a)) = \frac{1}{\Del_1} \left(
\det \left. \frac{\partial (f_1, f_2)} {\partial(y, \a)}\right|_{(0,0,0)}, 
-\det \left.\frac{\partial (f_1, f_2)} {\partial(x, \a)}\right|_{(0,0,0)}
\right)\a + O(\a^2).
\end{equation}
Similarly, as $\Del_2 \neq 0$, equations $f_1 = g = 0$ have a unique singular equilibrium solution:
\begin{equation}\label{eq:seq}
\displaystyle  (x_s(\a), y_s(\a)) = \frac{1}{\Del_2} \left(
\det \left. \frac{\partial (f_1, g)} {\partial(y, \a)}\right|_{(0,0,0)}, 
-\det \left.\frac{\partial (f_1, g)} {\partial(x, \a)}\right|_{(0,0,0)}
\right)\a + O(\a^2).
\end{equation}

At the bifurcation moment $\a = 0$ the linearization matrix (\ref{linear_s_eq}) at the singular saddle node has eigenvalues $\lambda_1 = f_{1x} \neq 0$ and $\lambda_2 = 0$. Upon a small perturbation, linearization matrices (\ref{linear_eq}) and (\ref{linear_s_eq}) of, respectively, the equilibrium and the singular equilibrium, will have real  eigenvalues $\lambda_1 = f_{1x}(0, 0, 0) + O(\a)$ and $\lambda_2 = O(\a)$, so they have either node or saddle type. The sign of the first eigenvalue at each point is given for small $\a$ by the sign of derivative $f_{1x}(0, 0, 0)$. 
The product of eigenvalues is equal to the determinant of the linearization matrix. At the equilibrium the determinant is:
$$
\det A_{EQ} = g(x^e(\a), y^e(\a), \a) (\Del_1 + O(\a)) = \Del_4 \a + O(\a^2),
$$
and at the singular equilibrium:
$$
\det A_{sEQ} = f_2(x^s(\a), y^s(\a), \a) (\Del_2 + O(\a)) = -\Del_4 \a + O(\a^2).
$$

Thus, the bifurcation occurs in the following way:
\begin{itemize}
    \item[] When $\Del_4 \a > 0$, the equilibrium is a node (stable if $f_{1x} < 0$ and unstable if $f_{1x} > 0$), and the singular equilibrium is a folded saddle;
    \item[] When $\Del_4 \a < 0$, the equilibrium is a saddle, and the singular equilibrium is a folded node.
\end{itemize}

\end{proof}

\begin{ex}\rm
Consider 
\begin{equation}\label{eq:2D_nf_ex}
\begin{cases}
(x - x^3)\dot x=y-x+\a\\
\dot y=y
\end{cases}.
\end{equation}
Then, for $\a = 0$ we have $f_1 = f_2 = g = 0$, i.e. a singular saddle-node of type II at the origin $(x, y) = (0, 0)$. By formulas (\ref{eq:determinants}), it follows that
$f_{1x} = -1$, $g_x = 1$, $\Del_1 = -1$, $\Del_2 = -1$ and $\Del_4=-1$, the conditions (\ref{noSingFD}) are fulfilled.
For small $\a$ this bifurcation point unfolds into an equilibrium EQ: $(x^e, y^e) = (\a, 0) + O(\a^2)$ and a singular equilibrium sEQ: $(x^s, y^s) = (0, -\a) + O(\a^2)$.
For $ \a > 0$ EQ is a stable node and sEQ is a folded saddle.
For $\a < 0$, EQ is a saddle and sEQ is a folded node.

\end{ex}

\subsection{L4: Cubic fold}
The cubic fold bifurcation occurs at point $M(0, 0)$ for $\a = 0$ if conditions $g = g_x = g_{xx} = 0$ are fulfilled at it. In addition, to keep the codimension of the problem equal to one and to construct a transversal parametric family, we assume the following inequalities to hold:
\begin{equation}\label{eq:cubic_fold_gen}
    f_1 \neq 0, \q g_y \neq 0, \q g_{xxx} \neq 0, \q \Del_3 \neq 0
\end{equation}

\begin{lem}\label{lem:cubic_fold}
{Assume that genericity (codimension one + transversality) conditions (\ref{eq:cubic_fold_gen}) are fulfilled at a cubic fold point $M$. Then in a generic unfolding point $M$ splits into a pair of simple folds, or disappears. 
}
\end{lem}

\begin{proof}
For system of equations $g = g_x = 0$ the Implicit Function Theorem are not fulfilled, because $\displaystyle \left.\frac{\partial(g, g_x)}{\partial(x, y)}\right|_{(0,0,0)} = 0$. Then we look for a solution of this system in the form 
$$
x = \beta_1 \alpha^{\delta_1} + o(\alpha^{\delta_1}), \q y = \beta_2 \alpha^{\delta_2} + o(\alpha^{\delta_2})
$$
for positive $\alpha$, and
$$
x = \beta_1 (-\alpha)^{\delta_1} + o((-\alpha)^{\delta_1}), \q y = \beta_2 (-\alpha)^{\delta_2} + o((-\alpha)^{\delta_2}) 
$$
for negative $\alpha$.
The equations take form 
\begin{equation}\label{c_fold_1}
\begin{array}{l}
\displaystyle
0 = g_y \beta_2 \alpha^{\delta_2} + g_\alpha \alpha + 
\frac{1}{6}g_{xxx}\beta_1^3 \alpha^{3\delta_1} + h.o.t.\\ 
\displaystyle
0 = g_{xy} \beta_2 \alpha^{\delta_2} + g_{x \alpha}\alpha +
\frac{1}{2} g_{xxx} \beta_1^2 \alpha^{2\delta_1} + h.o.t.
\end{array}
\end{equation}
and respectively
\begin{equation}\label{c_fold_2}
\begin{array}{l}
\displaystyle
0 = g_y \beta_2 (-\alpha)^{\delta_2} - g_\alpha (-\alpha) + 
\frac{1}{6}g_{xxx}\beta_1^3 (-\alpha)^{3\delta_1} + h.o.t.\\ 
\displaystyle
0 = g_{xy} \beta_2 (-\alpha)^{\delta_2} - g_{x \alpha}(-\alpha) +
\frac{1}{2} g_{xxx} \beta_1^2 (-\alpha)^{2\delta_1} + h.o.t.
\end{array}
\end{equation}
For them to be solvable it is required that $\delta_1 = 1/2$, $\delta_2 = 1$. Then for $\alpha > 0$ from (\ref{c_fold_1}) we have
\begin{equation}\label{c_fold_3}
\displaystyle  \beta_1^2 = \frac{2\Del_3}{g_y g_{xxx}}, \q 
\beta_2 = -\frac{g_\alpha}{g_y}
\end{equation}
and for $\alpha < 0$ from (\ref{c_fold_2}):
\begin{equation}\label{c_fold_4}
\displaystyle  \beta_1^2 = -\frac{2\Del_3}{g_y g_{xxx}}, \q 
\beta_2 = \frac{g_\alpha}{g_y}.
\end{equation}

Then, two simple folds exist for perturbations $\displaystyle \frac{2\Del_3}{g_y g_{xxx}} \alpha > 0$, and the cubic fold disappears, and no folds exist locally, when $\displaystyle \frac{2\Del_3}{g_y g_{xxx}} \alpha < 0$.

\end{proof}

\subsection{L5.  Singular equilibrium-fold}
The singular equilibrium-fold bifurcation occurs at point $M(0, 0)$ for $\a = 0$ if conditions $g = g_x = f_1 = 0$ are fulfilled at it. In addition, to keep the codimension of the problem equal to one and to construct a transversal parametric family, we assume the following inequalities to hold:
\begin{equation}\label{eq:eq_fold_gen}
    g_y \neq 0, \q g_{xx} \neq 0, \q \Del_2 \neq 0, \q \Del_3 \neq 0, \q \Del_5 \neq 0
\end{equation}

\begin{lem}\label{L:fold_sEQ}
{Assume that genericity (codimension one + transversality) conditions (\ref{eq:eq_fold_gen}) are fulfilled at a singular equilibrium-fold point $M$. Then in a generic unfolding point $M$ splits into a  simple folds and a simple singular equilibrium. 
}
\end{lem}

\begin{proof}
By the transversality condition $\Del_5 \neq 0$, system of equations $g = g_x = f_1 = 0$ has locally no solutions for $\a \neq 0$, then the singular equilibrium-fold disappears under such small perturbations.

The genericity condition $\Del_2 \neq 0$ implies that system of equations $f_1 = g = 0$ has locally a unique singular equilibrium solution for small $\alpha$:
\begin{equation} \label{eq_fold_sEQ}
\displaystyle  (x_s(\a), y_s(\a)) = \frac{1}{\Del_2} \left(
\det \left. \frac{\partial (f_1, g)} {\partial(y, \a)}\right|_{(0,0,0)}, 
-\det \left.\frac{\partial (f_1, g)} {\partial(x, \a)}\right|_{(0,0,0)}
\right)\a + O(\a^2).
\end{equation}

Also, the genericity condition $\Del_3 \neq 0$ implies that
a unique fold solution of system $g = g_x = 0$ exists for small $\a$:
\begin{equation} \label{eq_fold_fold}
\displaystyle  (x_f(\a), y_f(\a)) = \frac{1}{\Del_3} \left(
\det \left. \frac{\partial (g, g_x)} {\partial(y, \a)}\right|_{(0,0,0)}, 
-\det \left.\frac{\partial (g, g_x)} {\partial(x, \a)}\right|_{(0,0,0)}
\right)\a + O(\a^2).
\end{equation}

\end{proof}

 \section*{Funding}
This paper is a contribution to the project M7 (Dynamics of Geophysical Problems in Turbulent Regimes) of the Collaborative Research Centre TRR 181 ``Energy Transfer in Atmosphere and Ocean'' funded by the Deutsche Forschungsgemeinschaft (DFG, German Research Foundation) -- Projektnummer 274762653. The paper is also supported by the
grant of the Russian Science
Foundation 19-11-00280.

\end{document}